\documentclass[a4paper, 10pt]{scrartcl}
\pdfoutput=1

\title{Application of modal filtering to a Spectral Difference Method}
\author{ 
  Jan Glaubitz, 
  Philipp  \"Offner, 
  and Thomas Sonar
}
\date{2018}

\usepackage[utf8]{luainputenc}
\usepackage[english]{babel}
\usepackage{csquotes}

\usepackage[a4paper, textheight=700pt, textwidth=468pt, centering]{geometry}

\usepackage[plainpages=false,pdfpagelabels,hidelinks]{hyperref}
\makeatletter
\hypersetup{pdfauthor={\@author}}
\hypersetup{pdftitle={\@title}}
\makeatother

\usepackage{cite}

\usepackage{amsmath}
\allowdisplaybreaks
\usepackage{amssymb}
\usepackage{commath}
\usepackage{mathtools}

\usepackage{siunitx}

\usepackage{amsthm}
\theoremstyle{plain}
  \newtheorem{thm}{Theorem}[section]
  \newtheorem{lem}[thm]{Lemma} 
  \newtheorem{cor}[thm]{Corollary}
  \theoremstyle{definition}
  \newtheorem{ex}[thm]{Example}
\newtheorem{rem}[thm]{Remark}
  
\usepackage{color}
\usepackage{graphicx}
\usepackage[small]{caption}
\usepackage{subcaption}

\usepackage{pgfplots}
\pgfplotsset{compat=1.11}
\usepgfplotslibrary{external} 
\begingroup\expandafter\expandafter\expandafter\endgroup
\expandafter\ifx\csname pdfsuppresswarningpagegroup\endcsname\relax
\else
\fi

\usepackage{algorithm}
\usepackage[noend]{algpseudocode}

\usepackage{comment}

\usepackage{booktabs}
\usepackage{rotating}

\usepackage{calc}
\usepackage{xparse}

\newcommand{\subjclass}[1]{\textbf{MSC2010:} \ #1}
\newcommand{\keywords}[1]{\textbf{Keywords:} \ #1}

\NewDocumentCommand{\mat}{mo}{%
  \IfValueTF{#2}{%
    \underline{\underline{#1}}{#2}
  }{%
    \underline{\underline{#1}}\,
  }%
}
\newcommand{\diag}[1]{\operatorname{diag}\left(#1\right)}

\renewcommand{\d}{\operatorname{d}}

\renewcommand{\L}{\mathbf{L}}
\newcommand{\T}{\mathbb{T}} 
\newcommand{\ol}{\mathrm{o}}

\renewcommand{\epsilon}{\varepsilon}
\renewcommand{\phi}{\varphi}
\newcommand{\N}{\mathbb{N}}
\newcommand{\R}{\mathbb{R}}

\renewcommand{\bf}{\mathbf}
\renewcommand{\r}{\right}
\renewcommand{\l}{\left}

\newsavebox{\DelimiterBox}
\newlength{\DelimiterHeight}
\newlength{\DelimiterDepth}
\newsavebox{\ArgumentBox}
\newlength{\ArgumentHeight}
\newlength{\ArgumentDepth}
\newlength{\ResizedDelimiterHeight}

\begin{document}

\maketitle

\begin{abstract}
  We adapt the spectral viscosity (SV) formulation implemented as a modal filter to a 
Spectral Difference Method (SD) solving hyperbolic conservation laws. 
In the SD Method we use selections of different orthogonal polynomials (APK polynomials). 
Furthermore we obtain new error bounds for filtered APK extensions of smooth functions.
We demonstrate that the modal filter also depends on the chosen polynomial basis in the SD Method. 
Spectral filtering stabilizes 
the scheme and leaves weaker oscillations. 
Hence, the selection of the family of orthogonal polynomials on triangles and their specific modal 
filter possesses 
a positive influence on the stability and accuracy of the SD Method. 
In the second part, we initiate a stability analysis for a linear scalar test case with periodic 
initial condition 
to find the best selection of APK polynomials and their specific modal filter.
To the best of our knowledge, this work is the first that gives a stability analysis for a scheme 
with spectral filtering.
Finally, we demonstrate the influence of the underlying basis of APK polynomials in a well-known 
test case.
\end{abstract}



\section{Introduction}

In the field of Computational Fluid Dynamics (CFD), low-order methods are generally
robust and reliable and therefore employed in practical calculations. 
The main advantage of high-order methods towards low-order ones is the 
possibility of considerably more accurate solutions with 
the same computing cost, but unfortunately they are less robust and more
complicated. In recent years many researchers focus on this topic. 
There has been a surge of research activities to improve and refine 
high-order methods as well as to develop new ones with more favorable properties.
The Spectral Difference (SD) Method for simplex cells was first presented 
by Liu et al.  \cite{liu2006spectral}, and later extended by 
Wang et al.  \cite{wang2007spectral, sun2007high}. For one-dimensional, two-dimensional 
quadrilateral and three-dimensional hexahedral grids, the classical
SD Method is identical to the multi-domain staggered grid spectral element method proposed by Kopriva et al.
\cite{kopriva1996conservative, kopriva1995conservative}.  
Further contributions can be found inter alia in \cite{huang2005implicit, may2006spectral, van2008stability}.
There are many publications that show that the classical SD Method is closely 
related to the Discontinuous Galerkin (DG) Method 
and Spectral Volume  Method (or the same), see 
\cite{ allaneau2011connections, huynh2007flux, yu2014accuracy, wang2009unifying, van2007connection, huynh2011high, jameson2010proof} 
for example. Depending on  how the degrees of freedom (DOFs) are chosen, 
various method implementations have different numerical properties and efficiencies. 
But all of them have in common that they use piecewise continuous functions 
as approximation space for solutions. Here, we apply the classical orthogonal polynomials
on triangles in a Spectral Difference Method and focus on filtering techniques.
We don't use the nowadays more common Flux Reconstruction (FR) or
Correction Procedure via Reconstruction (CPR) methods, because we only want to investigate the pure influence of the orthogonal polynomials and 
their natural filters. In the FR/CPR approach, a correction term\footnote{With appropriate choices of 
correction terms, the FR framework recovers specific DG, SD, as well as spectral volume schemes.} is applied
which works at the interface and rectify the calculation in  in every step, 
for details see \cite{ranocha2015summation, huynh2007flux, vincent2011insights, 
huynh2014high, vincent2011newclass, ranocha2016enhancing, 
glaubitz2016enhancing} 
and references therein.
Furthermore, the SD Method on triangular grids is not stable, 
i.e. numerical solutions migth exceed all boundaries where they shouldn't. 
This allows us to oberve the influence of filter techniques and how they enable us 
to milder such instabilities of the scheme.
The SD Method combines the basic ideas of spectral methods and finite differences.
It directly reconstructs a flux polynomial based on fluxes on a given 
nodal set called flux points. Then the derivatives of the flux 
polynomial are used to update the solutions at the solution points. 
This approach prevents the use of quadrature rules like in a normal DG or Spectral Volume 
ansatz\footnote{There exists also quadrature-free implementations of DG, see for example  \cite{atkins1998quadrature}.}.
However,  the SD Method has stability issues especially for high orders or if discontinuities arise in the solution.
The root of the instability is that the nonlinear flux function is represented by an insufficient amount of points. 
This introduces aliasing errors \cite{kirby2006aliasing}. By increasing the dissipation we are able to mask such an aliasing problem. 
We apply the same approach like in spectral methods and use spectral viscosity to stabilize the calculation,
see \cite{maday1993legendre, tadmor1989convergence}. As suggested in \cite{ma1998chebyshev}, 
the spectral viscosity can be carried out within 
the spectral filtering framework, resulting in an  efficient
computational implementation. The Spectral Viscosity Method (SV) can be seen as a spectral method, 
but at each time step the numerical solution is filtered by an 
exponential filter which 
depends on the chosen set of orthogonal polynomials.\\
Meister et al.  \cite{meister2012application} apply the Proriol-Koornwinder-Dubiner (PKD) polynomials  in a
DG Method  and derive a relation between a modal 
filter for DG Methods on unstructured triangulation grids and the introduction of spectral viscosity to the scheme. 
The basic idea is to add a high order viscosity term, which 
is based on the Sturm-Liouville Operator of the polynomials
to the equation.\\
In this article, we consider 
the Spectral Difference Method as described 
in  \cite{wirz2014detecting}, but we extend it by the general classical orthogonal polynomials 
on triangles (APK polynomials)\footnote{The PKD or Dubiner polynomials are only one specific family of these polynomials.}
and their specific modal filters. With the differential operator of the APK polynomials, we
 show analogous 
to \cite{meister2012application} the close relation between SV and spectral filtering in the 
SD Method.\\ 
In a theoretical framework we prove some new upper bounds for the filtered APK expansion of smooth functions.
Here,
we generalize the theoretical results from \cite{meister2012application}, 
where only the properties of the filtered PKD expansion were analyzed.
Our viscosity term depends on the differential operator of the chosen polynomial set and 
so  
the exponential filter as well. Therefore
the selection of the
orthogonal polynomials and their natural filter have a positive impact
on the stability and accuracy of the method. By starting a 
stability analysis as presented
in \cite{huynh2007flux, van2008stability}
we get a better understanding of the influence of the orthogonal polynomials and their specific filters. 

This paper is organized as follows:\\
In Section \ref{Polynome} we will define 
these considered
polynomials and review some properties. 
We will repeat the main ideas on the Spectral Difference Method and explain our implementation in the next Section.
High order filters as the SV 
Method
as stabilizing technique are introduced in Section \ref{Filter}. 
Also, we prove
an error bound for the filtered APK expansion of smooth functions
and transfer the SV modification to our SD Method. 
In Section \ref{Stability} we start with the stability analysis for our SD Method.
A numerical experiment is presented before a conclusion and an outlook for future work finishes this paper.

\section{Appell-Proriol-Koornwinder polynomials and their properties}\label{Polynome}

In this Section we introduce 
the orthogonal polynomials under observation. 
We call the family of classical orthogonal polynomials on triangles 
Appell-Proriol-Koornwinder polynomials (APK polynomials). 
In \cite{offner2013spectral} the authors prove spectral convergence for the APK series. Their result gives us the theoretical foundation to use these polynomials  in a spectral method.\\
Let $P_n^{\alpha, \beta}(x)$ be the $n$-th Jacobi polynomial,
 $\T:=\{(x,y) \in \R^2| x\geq 0,y\geq 0, x+y\leq 1\}$ be the unit triangle and 
$h(x,y):=x^{\alpha-1}y^{\beta-1}(1-x-y)^{\gamma-\alpha-\beta} \; (
  \alpha, \beta,\gamma \in \N ,\gamma>\alpha+\beta-1$ and $\N=\{1,2,\cdots\})$ be the weight function, given in this domain.  
 For the sake of brevity, we introduce $p=\gamma-\alpha-\beta$ and $a_l=p+\beta+2l$.\\
Note that we only use $\alpha, \beta, \gamma \in \N$ for simplicity. In principle, $\alpha,\beta, \gamma \in \R^+_0$ is possible.

\begin{def}\label{DefPol}
The polynomials  $A_{m,l}(x,y)$, $m,l \in \N_0$, defined as
\begin{equation}\label{APK-polynomials}
 A_{m,l}(x,y):=P_m^{\alpha-1,a_l}(1-2x)P_l^{p,\beta-1}\left(\frac{2y}{1-x}-1 \right)(1-x)^l
\end{equation}
on $\T$ are called \emph{Appell-Proriol-Koornwinder polynomials} (APK polynomials). \\
If the triangle $\widetilde{\T}:=\{(x,y)\in \R^2|x\geq-1,y\geq -1, x+y\leq0 \} $ is used instead, then 
\eqref{APK-polynomials} transforms to
\begin{equation*}
 \widetilde{A}_{m,k}(x,y)=P_m^{\alpha-1,2l+\gamma-\alpha}(-x)P_{l}^{\gamma-\alpha-\beta,\beta-1}\left( \frac{2(y+1)}{1-x}-1 \right)\left(\frac{1-x}{2} \right)^l.
\end{equation*}
The special case $\alpha=\beta=1$ and $\gamma=2$ is called Proriol-Koornwinder-Dubiner polynomial (PKD polynomials)(see \cite{dubiner1991spectral, karniadakis2013spectral}) .
\end{def}
These are the classical orthogonal polynomials on triangles. Details of their properties can be found in  \cite{dunkl2014orthogonal, koornwinder1975two, suetin1999orthogonal}.
We start with the definition of the function spaces and norms we are dealing with. Then we summarize 
some of the  properties of APK polynomials and give a couple of estimates, 
which are needed in the sequel. The detailed proofs of these estimates can be found in \cite{offner2013spectral}.

\begin{def}\label{defNorm}
  Let $h(x,y)$ be the weight function on the triangle $\T$. 
  We denote by  $\L^2(\T,h)$ the Hilbert space with the inner product 
\begin{equation*}
(u,v) :=\int_{\T} h(x,y) u(x,y)v(x,y) dxdy
\end{equation*} 
which induces the weighted norm 
\begin{equation*}
 ||u||_{\L^2(\T,h)}:=\left( \int_{\T} h(x,y) |u(x,y)|^2 dxdy\right)^\frac{1}{2}.
\end{equation*}
We introduce by $H^m(\T,h)$ a \emph{weighted Sobolev space}. Precisely we set
\begin{align*}
 H^m(\T,h):=\{v\in \L^2(\T,h)&\text{:for each non-negative multi-index $\sigma$ with} \\ 
                           &\text{ $|\sigma| \leq m$, the distributional derivative $D^{\sigma}v$ }\\
                           &\text{belongs to $\L^2(\T,h)$}\}.
\end{align*}
 The space is endowed with the norm
\[
 ||v||_{H^m(\T,h)}:=\left(\sum\limits_{|\sigma| \leq m} ||D^\sigma v||^2_{\L^2(\T,h)} \right)^\frac{1}{2}.
\]
\end{def}

The APK polynomials are characterized 
by the integers $m$ and $l$ as can be seen in 
the definition. The degree of an APK polynomial then is $m+l$.
Furthermore the APK polynomials are orthogonal in $\L^2(\T,h)$, i.e.
\begin{equation*}
 (A_{m,l},A_{n,s})=\delta_{m,n}\delta_{l,s}\frac{1}{(2l+\gamma-\alpha)(2(m+l)+\gamma)\kappa_{l,m}^2}
\end{equation*}
with 
\begin{equation}\label{kappa}
\kappa_{l,m}:=\sqrt{\frac{(l+\beta)_p m(m+a_l)_{\alpha}}{(l+1)_p(m)_{\alpha} 
(m+a_l)} },
\end{equation}
where
\begin{equation*}
 (\xi)_{0}= 1, (\xi)_j=\prod\limits_{i=0}^{j} (\xi+i-1)   \text{ for } j\in \N
\end{equation*}
 is the usual Pochhammer symbol.
See \cite[Chapter X. p.288ff]{suetin1999orthogonal}.

Another important property of the APK polynomials can be found in the fact that they are solutions to the eigenvalue problem
\begin{equation}\label{Eigenwertgleichung}
  DA_{m,l}=\lambda_{m,l}A_{m,l} 
\end{equation}
where the differential operator $D$ is given by
\begin{equation*}
(x^2-x)\frac{\partial^2 }{\partial x^2} +2xy \frac{\partial^2 }{\partial x \partial y}+(y^2-y) \frac{\partial^2 }{\partial y^2}
+[(\gamma+1)x-\alpha ]\frac{\partial }{\partial x}
+[(\gamma+1)y-\beta ]\frac{\partial }{\partial y}, 
\end{equation*}
and the eigenvalue $\lambda_{m,l}=(m+l)(m+l+\gamma)$, see \cite[p.46]{dunkl2014orthogonal}.
Note that this eigenvalue equation is 
\textit{not} a Sturm-Liouville problem and also that the differential operator is not self-adjoint 
for all choices of $\alpha, \beta, \gamma$. 
It is only self-adjoint in case of 
$\alpha=\beta=1$ and $\gamma=2$.
But in fact, the operator is, what is called, \emph{potentially self-adjoint} in $\mathring{T}$, see \cite[p.136]{suetin1999orthogonal}. Hence 
there exists a positive $C^2$- function $g$, $(x,y) \mapsto g(x,y)$, so that
$gD$ is self-adjoint in $\mathring{\T}$. This will be important in Section \ref{Filter}, when we use the differential operator in the viscosity term. 
In the following Lemma we summarize 
two estimates which will be needed in the proof of Theorem \ref{Theorem1}. The 
proofs can be found in \cite{offner2013spectral}.

\begin{lem}\label{LemmaNorm+Betrag}
 The following norm estimate holds for APK polynomials,
\begin{equation}\label{1druchdieNormAPK}
 \frac{1}{||A_{m,l}||_{L^2(\T,h)}} \leq 2(m+l+\gamma) \kappa_{l,m} 
\end{equation}
with $\kappa_{l,m}$ given by (\ref{kappa}).\\

Furthermore let $(x,y) \in \mathring{\T}$. For all $l,m \in \N_0$  the following estimate holds,
\begin{equation}\label{BetragAPKimInneren}
 |A_{m,l}(x,y)| \leq \frac{\tilde{D}(x,y)}{(2l+\beta+p)^{\frac{1}{4}}(2(m+l)+\gamma)^{\frac{1}{4}} \kappa_{l,m} },
\end{equation}
where
\begin{equation*}
 \tilde{D}(x,y)= \frac{D}{2(1-x-y)^{\frac{p}{2}+\frac{1}{4}}y^{\frac{1}{4}+\frac{\beta-1}{2}}
x^{\frac{1}{4}+\frac{\alpha-1}{2}}(1-x)^{\frac{1}{4}}}
\end{equation*}
and $D < 144$ is a positive constant.
 The value of $A_{m,l}(1,0)$ is
\begin{equation}\label{|A_{m,0}|}
 |A_{m,l}(1,0)|= \begin{cases}
                  \binom{m+\gamma-\alpha}{m} \quad \text{      if    } l=0\\
                 \hspace{0.5cm}     0 \qquad \text{      else }
                 \end{cases}.
\end{equation}
\end{lem}
\begin{rem}
  Similar estimates also hold for the other edges of $\T$ and can be seen in 
  \cite{offner2013spectral}.
\end{rem}

\section{An extended Spectral Difference Method}\label{SD Method}

In this paper we consider two-dimensional hyperbolic conservation laws of the form 
\begin{equation}\label{hypGleichung}
 \frac{\partial }{\partial t}{\boldsymbol{u}}({\boldsymbol{x}},t)=-\nabla_{\boldsymbol{x}} F({\boldsymbol{u}}({\boldsymbol{x}},t)), \quad ({\boldsymbol{x}},t) \in \Omega\times \R_0^+,
\end{equation}
where $\Omega \subset \R^2$ is an open polygonal domain and ${\boldsymbol{u}}({{\boldsymbol{x}}},t)\in \R^n$.
Furthermore, initial conditions ${\boldsymbol{u}}({\boldsymbol{x}},0)={\boldsymbol{u}}_0(\boldsymbol{x})$ and appropriate boundary conditions
are assumed to be given. 

Hence, let $\mathcal{T}_h$ be a conforming triangulation of the closure $\overline{\Omega}$ of the computational domain and let $\mathbb{P}^h$ be the piecewise polynomial space 
defined by $\mathbb{P}^h=\{  \ v_h|_{\tau_i}\in \mathcal{P}^N (\tau_i), \; \ \forall \tau_i\in \mathcal{T}_h \}$, where $\mathcal{P}^N(\tau_i)$ denotes the 
space of all polynomials on $\tau_i$ of degree less than or equal to $N$.

The classical Spectral Difference Method, which has been proposed by Liu \cite{liu2006spectral} in 2006, can be seen as a FR
or collocation method. The basic idea of this method is to discretize the right hand side of the underlaying conservation law
(\ref{hypGleichung}) 
at certain solution points $\bf{x}_j$ in each cell. Then, the resulting ODEs (in $t$) at each ${\bf{x}}_j$ can be solved by an arbitrary explicit time-stepping scheme. 
Here we used the $4$-th order low storage Runge-Kutta scheme defined by Carpenter and Kennedy, see \cite{carpenter1994fourth}.
We say that a scheme is of \textit{order} $N+1$, when it is exact for $\boldsymbol{u} \in \left[ \mathcal{P}^N \right]^n$. 
Since the derivate of the flux $F$ is applied to update 
$\boldsymbol{u}$, 
one needs to be exact for $F\in \left[ \mathcal{P}^{N+1} \right]^{n\times 2}$.
We approximate the flux $F$ in each element  $\tau_i$ of our triangulation $\mathcal{T}_h$ using the basis polynomials $\phi_k$ in the 
following way
\begin{equation}\label{Fluss}
\begin{pmatrix} F_1({\boldsymbol{x}},t) \\ F_2({\boldsymbol{x}},t) 
\\ \end{pmatrix}=\sum\limits_{k=1}^{K_F} 
\begin{pmatrix} \hat{F}_{k,1}(t) \\ \hat{F}_{k,2}(t) \\ \end{pmatrix} 
 \phi_k({ \boldsymbol{x}}).
\end{equation}

The classical approach uses Lagrange polynomials $L_k$ with corresponding coefficients 
$\hat{F}_{k,{\nu}}(t)=F_{\nu}({{\bf x}_k },t)$ for $\nu = 1,2$
, where ${\bf{x}}_k$ are chosen flux points.
If we want to achieve a method of order $N+1$, we need the reconstruction of the solution $\boldsymbol{u}$ to lie in $\left[ \mathcal{P}^N \right]^n$ 
and the reconstruction of the flux $F$ to lie in $\left[ \mathcal{P}^{N+1} \right]^{n\times 2}$. Hence we need $K_s := \frac{(N+1)(N+2)}{2}$ solution 
points $\bf{x}_j$ and $K_F := \frac{(N+2)(N+3)}{2}$ flux points $\bf{x}_k$. 
That means points where the flux value is computed. 
Here we apply the  classic orthogonal polynomials on triangles $\phi_k=A_{m,l}$ (in lexicographic order $k$). 
In order to adapt the polynomial basis $\phi_k$ to every  triangular cell $\tau_i$, we introduce an orientation-preserving affine transformation $T_i$
which maps an arbitrary triangle $\tau_i$ to the standard triangle $\T$ of the orthogonal basis $\phi_k$, see Figure \ref{Trafo}. 
\begin{figure}[!htb]
\centering
  \includegraphics[width=0.8\textwidth]{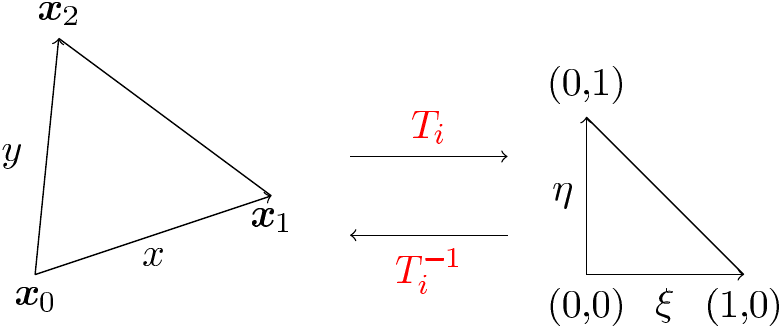}
  \caption{transformation of an arbitrary triangle to the standard element }
  \label{Trafo}
\end{figure}
By linearity of the differential operator, inserting \eqref{Fluss} into \eqref{hypGleichung} leads to
\begin{equation*}
{ \boldsymbol{u}}_t({\boldsymbol{x}},t)=-\sum\limits_{k=1}^{K_F} 
 \begin{pmatrix} \hat{F}_{k,1}(t) \\ \hat{F}_{k,2}(t) \\ \end{pmatrix}\cdot \nabla_{\boldsymbol{x}} \phi_k(T_i({\boldsymbol{x}})). 
\end{equation*}

Applying the transformation and the chain rule to 
\begin{equation*}
\nabla_{\boldsymbol{x}} \phi_k(T_i({\boldsymbol{x}}))= J_{T_i} \nabla_{\boldsymbol{\xi}}\phi_k(T_i(\boldsymbol{x})), 
\end{equation*}
finally leads to the 
universal update scheme 
\begin{equation}\label{update}
 {\boldsymbol{u}}_t({\boldsymbol{x}},t)=-\sum\limits_{k=1}^{K_F} 
 \begin{pmatrix} \hat{F}_{k,1}(t) \\ \hat{F}_{k,2}(t) \\ \end{pmatrix}\cdot  J_{T_i} \nabla_{\boldsymbol{\xi}}\phi_k(T_i(\boldsymbol{x})).
\end{equation}
where for each cell only the Jacobian $J_{T_i}$ has to be stored.
%
  For the stability analysis in Section \ref{Stability} we will derive benefit from the following matrix representation 
  of (\ref{update}). If we denote the vector $\bigl( {\boldsymbol{u}}_t({\bf{x}}_j,t) \bigr)_j$ at the $K_s$ solution
  points ${\bf x}_j$ by $\frac{ \d {\bf u}^s}{ \d t}(t)$ and $J_{T_i}$ by $\begin{pmatrix}
    \xi_x & \xi_y \\ 
    \eta_x & \eta_y
  \end{pmatrix}$, the universal update scheme reads
  \begin{align}\label{schemeMatrix1}
    \frac{ \mathrm{d} \mathbf{u}^s }{ \mathrm{d} t }(t) = 
    & - \Bigl( \sum_{k=1}^{K_F} \hat{F}_{k,1}(t) \bigl[ \xi_x \partial_\xi \phi_k(T_i({\bf x}_j)) 
+ \xi_y \partial_\eta \phi_k(T_i({\bf x}_j)) \bigr] \Bigr)_j \nonumber \\
    & - \Bigl( \sum_{k=1}^{K_F} \hat{F}_{k,2}(t) \bigl[ \eta_x \partial_\xi \phi_k(T_i({\bf x}_j)) 
+ \eta_y \partial_\eta \phi_k(T_i({\bf x}_j)) \bigr] \Bigr)_j \nonumber \\
    = & - \xi_x D_\xi \hat{\mathbf{F}}_1(t) - \xi_y D_\eta \hat{\mathbf{F}}_1(t) - \eta_x 
D_\xi 
    \hat{\mathbf{F}}_2(t) - \eta_y D_\eta \hat{\mathbf{F}}_2(t)
  \end{align} 
  with 
  \begin{equation}\label{diffM}
    \begin{aligned}
  & D_\xi = \bigl( \partial_\xi \phi_k(T_i(\mathbf{x}_j)) \bigr)_{j,k} \; , \; \hat{\mathbf{F}}_1(t) 
= \bigl( \hat{F}_{k,1}(t) \bigr)_k \; \text{ and } \\
  &  D_\eta = \bigl( \partial_\eta \phi_k(T_i({\bf x}_j)) \bigr)_{j,k} \; , \; \hat{\mathbf{F}}_2(t) 
= \bigl( \hat{F}_{k,2}(t) \bigr)_k.
    \end{aligned}
\end{equation}

\begin{minipage}{0.39\textwidth}
\includegraphics[width=0.7\textwidth]{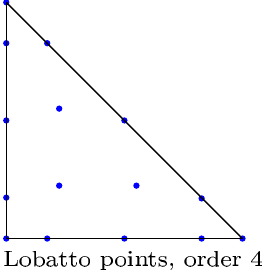}\label{Lobatto}
\end{minipage}
\begin{minipage}{0.59\textwidth}
In spite of the Lagrange reconstruction approach where the flux coefficients are directly known we have to compute 
$\hat{F}_{k, { \nu} }(t)$, 
in every time step. We choose an interpolation approach, i.e. we interpolate the values of the flux function at certain flux points $\{{\bf{x}}_k \}$ from the given flux $F$, which will be specified later on. 
The nodal set $\{ {\bf x}_k \}$ is chosen as the set of two dimensional Lobatto points on a triangle, see Figure \ref{Lobatto}, proposed by Blyth
and Pozrikidis
\cite{blyth2006lobatto} since they are easy to implement and have good interpolation properties, 
as 
for instance a low condition number. 
\end{minipage}

The condition number
of the Vandermonde matrix $\mathcal{V}$ in the interpolation approach 
 depends on the parameters $\alpha, \beta$ and $\gamma$ of the APK polynomials, 
see Table \ref{Konditionszahl}.
The numbers in the braces describe the parameters $(\alpha, \beta, \gamma)$ of the APK-polynomials. 
We realize that Lagrange polynomials lead to completely bad conditioned basis compared to 
any of the considered APK families. 
Regarding this, one should note that the coefficients of the Lagrange reconstruction are directly given 
by the values of the flux at the flux points. 
This values however are obtained from the values of the flux at the solution points by Lagrange interpolation. 
Additional numerical errors might arise.

\begin{table}[ht]
\begin{center}
 \begin{tabular}{c|c|c|c|c|c|c|c|c|c}
 $N$ &  $2$ & $3$ &  $4$& $5$& $6$& $7$ & $8$& $9$ & $10$  \\ \hline \hline
 Lagr. poly. &  $24$ & $270$ & $2023$ & $10^4 $ & $10^5$ & $7\cdot10^5$ & $5\cdot 10^6  $ & $3\cdot 10^7$ & $2\cdot 10^8$  \\
 $(0.1,0.1,0.2)$ &$ 72 $ & $23$ & $103 $ & $162 $ & $ 305 $ & $315 $ & $ 429$ & $ 577$ & $ 681 $ \\
 $(0.5,0.5,1)$& $12 $ & $17$   & $ 30$ & $42$  & $52 $ & $ 71$ & $82 $ & $100 $ & $ 121$  \\
 $(1,1,2)$           &  $12$   & $20 $ & $ 39 $ & $53 $ & $71$ & $94 $ & $121 $ & $ 151$ & $196 $ \\
  $(1,1,6)$ & $29$ & $ 117$ & $391 $ & $ 973 $ & $2365 $ & $4484 $ & $8008 $  & $ 10^4$ & $2\cdot10^4 $ \\
 $(2,2,5)$ & $17$ & $44 $ & $96 $ & $161 $ & $244 $ &  $450$ & $ 662 $  & $ 853 $ & $1341 $ \\
 $(3,3,20)$ & $ 28 $ & $131$ & $416 $ & $929 $ & $2310 $ & $ 4711 $ & $ 10^4 $ & $2\cdot 10^4 $ & $3\cdot 10^4 $\\
 $(10,10,20)$ & $306 $  &  $2307$ &  $2 \cdot 10^4  $ &  $8 \cdot 10^4  $ &  $4 \cdot 10^5 $& $2 \cdot 10^6  $ & $8 \cdot 10^6 $ &  $2 \cdot 10^7 $ & $8 \cdot 10^7  $    \\
 $(1,2,3)$ & $11 $  &  $33$ &  $66  $ &  $99 $ &  $147 $& $209 $ & $295 $ &  $409 $ & $532 $   \\
 $(2,1,3)$ & $17$ & $34$  & $66 $ & $123$  & $178 $ & $ 243$ & $388 $ & $446 $ & $ 570 $ 
\end{tabular} 
\caption{condition numbers $\kappa_N$}\label{Konditionszahl}
\end{center}
\end{table}

Additionally, enough flux points lie on the edges of an element to ensure global conservation \cite{wirz2012spektrale}, 
which is then 
realized by replacing the original flux $F$ at flux points ${\bf{x}}_k$ at the edges by a numerical flux $F^{num}$ whose normal component is computed from a numerical flux function $H$, i.e.
$F_k^{num}\cdot {\bf{n}}=H({\boldsymbol{u}}_{-} ({\bf{x}}_k,t),{\boldsymbol{u}}_{+}({\bf{x}_k},t),\bf{n})$. 
The whole flux is used in the scheme, so 
that one needs a second condition to determine the numerical flux. As in \cite{wang2007high}, we enforce $F^{num}$ to maintain the same tangential component as the original flux, i.e. 
$F_k^{num}\cdot {\bf{t}}=F({\boldsymbol{u}}({\bf{x}}_k,t))\cdot \bf{t}$, 
which uniquely defines the numerical flux $F_k^{num}=(F_{k,1}^{num},F_{k,2}^{num})^T$ at each edge point ${\bf{x}}_k$.
Now, the coefficients $ \hat{F}_{k,{\nu}}(t)$ 
in equation \eqref{update} can be computed by solving the system of equation ${\bf{F}}_{\nu}(t)=\mathcal{V}\cdot \hat{\bf{F}}_{\nu}(t)$
for each component $\nu=1,2$, where $\hat{\bf{F}}_{\nu}=\left( \hat{F}_{k,\nu}(t) \right)_k, \; \mathcal{V}=\left(\phi_k(\bf{x}_j) \right)_{j,k}$ is the Vandermonde matrix and 
$\hat{\bf{F}}_{\nu}(t)=(F_{k,\nu})_k$ is given by 
\begin{equation}\label{numerical flux}
 F_{k,{\nu}}=\begin{cases}
            F_{\nu}({\boldsymbol{u}}({\bf{x}}_k,t)), \quad \text{ if } {\bf{x}}_k \in \mathring{\tau_i} \\ 
            F_{k,\nu}^{num}, \hspace{1.0cm} \quad \text{ if }  {\bf{x}}_k \in \partial \tau_i.
           \end{cases}
\end{equation}

If we apply this to (\ref{schemeMatrix1}) the universal update scheme reads
\begin{align}\label{schemeMatrix2}
  \frac{ \d \mathbf{u}^s }{ \d t}(t) = - ( \xi_x D_\xi + \xi_y D_\eta ) \mathcal{V}^{-1} {\bf 
F}_1(t) - ( \eta_x D_\xi + \eta_y D_\eta ) \mathcal{V}^{-1} {\bf F}_2(t)
\end{align}
in a matrix representation. In Section \ref{Stability} we will use this result to obtain the form
\begin{align*}
  \frac{ \d }{ \mathbf{u}^s}{ \d t}(t) = S {\bf u}^s(t)
\end{align*}
with a matrix $S$ for a suitable test case.
In the numerical examples in Section \ref{Numerical Test} we will use the 
Godunov flux
\begin{equation*}
 H(u_{-},u_{+},{\bf{n}})=\begin{cases}
                      \min \limits_{u_{-}\leq u\leq u_{+}} F(u)\cdot {\bf{n}}, \qquad \text {if }  u_-\leq u_+,\\
                     \max \limits_{u_{+}\leq u\leq u_{-}} F(u)\cdot {\bf{n}}, \qquad \text{else}, \\                      
                     \end{cases}
\end{equation*}
in the scalar case for the numerical flux function. The Godunov 
flux is an Upwind flux. For the linear transport equation  we 
obtain the numerical flux 
\begin{equation*}
 H(u_{-},u_{+},{\bf{n}})=\begin{cases}
                       {\bf{(a\cdot n)}} u_{-}, \qquad  {\bf{(a\cdot n)}}\geq 0,\\
                                  {\bf{(a\cdot n)}} u_{+}, \qquad  {\bf{(a\cdot n)}} <0.\\         
                     \end{cases}
\end{equation*}

We will apply this numerical flux function in the stability analysis in Section 
\ref{Stability} as well.
For a more detailed explanation of the Spectral Difference Method and the 
extension we advise \cite{wirz2012spektrale}.
Yet, to at least highlight a significant advantage of the Spectral 
Difference Method, we observe the experimental order of convergence (EOC) in 
the $L_\infty$-norm for the linear advection equation 
\begin{equation*}
\frac{\partial }{\partial 
t}{\boldsymbol{u}}({\boldsymbol{x}},t)= -\frac{\partial }{\partial 
x}{\boldsymbol{u}}({\boldsymbol{x}},t) - \frac{\partial }{\partial 
y}{\boldsymbol{u}}({\boldsymbol{x}},t) \text{ with } (\boldsymbol{x},t) \in 
[-1,1]^2 
\times [0,0.5]
\end{equation*}
and for the smooth initial condition $u(\boldsymbol{x},0)=\sin \pi(x+y)$.
In Table \ref{Transport1}, this is done with respect to both the number of 
triangles (EOC(k)) and the degree of the polynomial approximation in each 
triangle (EOC(N)). For these approximations the parameters 
$(\alpha,\beta,\gamma)=(2,2,5)$ were chosen. Note that for unstructured 
triangulations some measurement $h$ for the size of the triangles should be 
used for the rate of convergence. Here for example, $h$ was the biggest volume 
of the triangles. At the same time however, in this test, $h$ was directly 
inversely proportional to $k$, i.e. $h = const \cdot \frac{1}{k}$. 

\begin{table}[!htb] 
\centering
\begin{tabular}{ r | c | c | c | c }
  $k$ & 68 & 272 & 1088 & 4352 \\ \hline
  $h$ & 0.0497 & 0.0124 & 0.0031 & 0.0008
\end{tabular}
\end{table}

Since EOC(h) is equals to EOC(k). 
For the sake of simplicity, we therefore observed 
convergence in the number of triangles $k$ and not their size $h$. 

\begin{table}[!htb]
  \centering
 \begin{tabular}{c|c|c|c|c| c| c| c| c| c| }
 
 $N$  & $k$ & $L_{\infty}$-error  & EOC(k) & EOC(N) & 
$L_1$-error & 
$L_2$-error  & time   \\ \hline \hline
 1 &  68 & 1.226072e+00 & & & 5.234411e-01 & 3.483826e-01 & 8 
\\   
2  &68& 2.999127e-01 & & 2.03 & 1.106267e-01  & 8.451124e-02 & 
15 \\ 
3  &68& 6.865602e-02& & 3.63 & 1.461619e-02 & 1.290576e-02 & 28 
\\ 
4  &68& 1.358577e-02& & 5.63 & 1.654693e-03 & 1.770760e-03 &  48 
   \\ 
5  &68& 1.583057e-03& & 9.63 & 1.353175e-04  & 1.657258e-04 & 78 
    \\ 
\hline
1  &272& 4.154960e-01& 0.78 & & 1.638561e-01 & 
1.079845e-01 & 33  \\ 
2  &272& 4.481682e-02& 1.37 & 3.21 & 1.738056e-02 & 
1.300429e-02 & 62 \\ 
3  &272& 5.600234e-03& 1.80 & 5.12 & 1.203168e-03 & 
1.050415e-03 & 110  \\ 
4  &272& 5.135337e-04& 2.36 & 8.30 & 6.683444e-05 & 
6.523343e-05 & 190    \\
5  &272 & 2.973202e-05& 2.86 & 12.76 & 3.036868e-06 & 
3.352791e-06 & 312    \\ 
\hline
1  &1088& 1.141089e-01& 0.93 & & 4.514174e-02 & 
2.951796e-02 & 129  \\ 
2  &1088& 6.801320e-03& 1.36 & 4.06 & 2.623593e-03 & 
1.942249e-03 & 243   \\ 
3  &1088& 4.438368e-04& 1.82 & 6.73 & 1.007609e-04 & 
8.647829e-05 & 437   \\ 
4  &1088& 1.793560e-05& 2.41 & 11.15 & 3.016446e-06 & 
2.706330e-06 & 761   \\ 
5  &1088& 5.468877e-07& 2.88 & 15.64 & 8.022185e-08 & 
7.697035e-08 & 1247  \\ 
\hline
1  &4352& 3.105485e-02& 0.93 & & 1.174054e-02 & 
7.637636e-03 & 526    \\ 
2  &4352& 1.458704e-03& 1.11 & 4.41 & 4.029470e-04 & 
3.134510e-04 & 860    
\\ 
3  &4352& 4.706511e-05& 1.61 & 8.46 & 8.869529e-06 & 
7.971082e-06 & 1766 \\ 
4  &4352& 1.320516e-05& 0.22 & 4.41 & 2.170014e-07 & 
3.488747e-07 & 3079 \\ 
5  &4352& 2.180251e-06& -0.99 & 8.07 & 1.510489e-08 & 
4.853751e-08 & 5144\\ 
\end{tabular} \caption{$ (\alpha,\beta,\gamma)=(2,2,5 ),\;$  
$t=0.5s$}\label{Transport1}

\end{table}

Table \ref{Transport1} clearly indicates the rate of convergence 
to be considerably higher when increasing the polynomial degree instead of 
refining the triangulation. 
For sufficiently smooth solutions, schemes using a polynomial 
approximation, in fact, often provide significant higher rates of convergence 
than classical ones, where low degrees (in particular 0) are used.
Note that in Table \ref{Transport1}, the same behavior can be observed for 
for the $L_1$- and $L_2$-norm.

\section{The filtering process}\label{Filter}
In the conservation law, 
discontinuities may arise in the solution. Using a series expansion 

\begin{equation}\label{Entwicklung}
 P_N u(x,y)=\sum\limits_{\substack{l+m\leq N \\l,m\ \in \N_0}} 
 \tilde{u}_{m,l}A_{m,l}(x,y); \quad 
 \tilde{u}_{m,l}=\frac{(u;A_{m,l})_{\L^2(\T,h)}}{(A_{m,l};A_{m,l})_{\L^2(\T,h)}}
\end{equation}

to approximate the solution leads 
to spurious oscillations in the vicinity of discontinuities (called Gibbs phenomenon). 
The oscillations occur in the approximated solution because the high coefficients of 
the series expansion turns slowly to zero, see \cite{gottlieb1977numerical, gottlieb1997gibbs}.\\
The aliasing error effects stability problems in the SD Method. 
The root of the instabilities is that the nonlinear flux function is represented by an 
insufficient amount of points. This introduces aliasing errors. These stability problems display especially near the oscillations. 
The amplitudes increase in time exceeding the error expected from the pure Gibbs phenomenon,  see Section
\ref{Numerical Test}. A higher inherent dissipation is able to mask such an aliasing problem. By
adding spectral viscosity to the equation we increase the dissipation. 
\\
To remedy the
Gibbs' effect, different approaches can be found in the literature 
\cite{gelb2006robust, gottlieb2001spectral, gottlieb1997gibbs, vandeven1991family}. A common approach is to use a modal filter, which appeals directly on the high-order coefficients of the series expansion. 
While multiplying a filter function to the high Fourier coefficients, the series loses their approximation properties. We prove new error bounds for the filtered APK series expansion 
of smooth function. We generalize the results of \cite{meister2012application} to all series expansion with classic orthogonal polynomials on triangles. 
Global filtering in each cell degrades 
the order of accuracy, see \cite{meister2012application}. 
Therefore we apply a well-known jump indicator \cite{persson2006sub}
and transfer the SV modification to the SD Method.

\subsection{Modal filters}
A modal filter appeals directly to 
the coefficients of the series expansion. 
\begin{def}\label{defFilter}
 For an integer $p\geq 1$ we define a \emph{filter} of order $p$ as a real function 
$\sigma\in C^{p-1}([0,1])$ with the properties
 \begin{equation}\label{Filtereigenschaften1}
 \begin{aligned}
  \sigma(0)&=1, \\
  \sigma^{(k)}(0)&=0, \qquad 1\leq k \leq p-1.
 \end{aligned}
 \end{equation}
\end{def}
Additionally to the properties \eqref{Filtereigenschaften1} many authors demand 
the following condition for a filter of order $p$:
\begin{equation}\label{Filtereigenschaften2}
 \sigma^{(k)}(1)=0 \qquad 0\leq k \leq p-1,
\end{equation}
compare \cite{gottlieb1997gibbs, vandeven1991family}.  In \cite{hesthaven2008filtering}, Hesthaven and Kirby require 
in addition  $\sigma \in C^p([0,1])$. A very popular filter is the exponential filter
\begin{equation}\label{Exponentialfilter}
 \sigma(\eta)=\exp(-\alpha \eta^p),
\end{equation}
where the filter strength $\alpha$ yields $\exp(-\alpha)$ in the range of the 
machine accuracy. 
We already mentioned that 
there are various partial results for the approximation property of a filtered 
series expansion.
For instance, in \cite{vandeven1991family} Vandeven has analyzed the filtered 
Fourier expansions of a piecewise smooth function and later expanded his result 
to filtered Chebyshev series. On the other hand, Kirby and Hesthaven 
investigate the approximation 
property of the filtered Legendre expansion for sufficiently smooth functions 
in \cite{hesthaven2008filtering}. 
\cite{meister2012application} expands the investigation to two-dimensional basis functions 
(PKD polynomials). 
Finally, we now complete 
the investigation from \cite{meister2012application}, by proving an error bound  
for the filtered APK partial sum. 
Again, we stress that the PKD polynomials are just a special case of the APK 
polynomials.
Here we prove 
the most general case and consider the filtered APK expansions 
for 
sufficiently smooth functions 
$u:\T\to \R$,
\begin{equation*} 
 u_N^{\sigma}(x,y)=\sum\limits_{l+m\leq N} \sigma \left(\frac{l+m}{N}\right) \tilde{u}_{m,l} A_{m,l}(x,y), \qquad N \geq 1,
\end{equation*}
with coefficients $\tilde{u}_{m,l}$ given by \eqref{Entwicklung}.

\begin{thm}\label{Theorem1}

 Let $u \in H^{2k} (\T,h) \cap C(\T), k \in \N$, $h(x,y)=x^{\alpha-1}y^{\beta-1}(1-x-y)^p $ with $\alpha, \beta \in \N$ and $ p \in \N_0 $ and let $\sigma $  be a modal 
 filter of order $2k-1$, with the additional condition  $\sigma\in C^{2k-1}([0,\varepsilon))$ in an interval $[0,\varepsilon)\subset [0,1]$, $\varepsilon>0$. 
Furthermore let $k>\max \left\{\ \frac{3}{4}+\frac{3}{4}\alpha+\frac{p}{2},\frac{5}{4}+\frac{p+\beta}{2}, \frac{1}{4}+\frac{3}{4}\alpha+\frac{\beta}{2} \right\} $.\\
Then we obtain the pointwise error bounds with  constants $K_1-K_6$:
\begin{enumerate}
 \item If $(x,y) \in  \mathring{\T}$,  it is
 \begin{equation*}\label{Inneredreieckfilter} 
 |u(x,y)-u_N^{\sigma}(x,y)|\leq  K_1 \frac{1}{N^{2k-\frac{7}{4}}}.
 \end{equation*}
\item 
On the left edge $ [0,y] $ with $y \in [0,1]$ we obtain
 \begin{equation*}\label{linkekantefilter} 
 |u(0,y)-u_N^{\sigma}(0,y)|\leq \begin{cases}
                               K_2 \frac{1}{N^{2k-\frac{3}{2}\alpha-p-\frac{1}{2}}}, & \text{ for } p> \beta-1, \\
                               K_3 \frac{1}{N^{2k-\frac{3}{2}\alpha-\beta+\frac{1}{2}}}, &\text { for } p\leq \beta-1. \\
                               \end{cases}
 \end{equation*}
\item 
On the edge $[x,0] $ with $x\in (0,1)$ the error is 
\begin{equation*}\label{unterkantefilter}
 |u(x,0)-u_N^{\sigma}(x,0)| \leq  K_4 \frac{1}{N^{2k-\frac{3}{4}-\beta}}.
\end{equation*}
\item On the hypotenuse $[x,1-x]$ with $x\in (0,1)$ pertain
\begin{equation*}\label{Hypothenusefilter}
 |u(x,1-x)-u_N^\sigma(x,1-x)|\leq K_5 \frac{1}{N^{2k-\frac{7}{4}-p}}.
\end{equation*}
\item
In the point $(1,0)$ we preserve
\begin{equation*}\label{Punktfilter}
 |u(1,0)-u_N^\sigma(1,0)|\leq K_6 \frac{1}{N^{2k+\alpha-\gamma -\frac{3}{2}}},
\end{equation*}
\end{enumerate}
\end{thm}

\begin{proof}
We start the proof similar to \cite[Theorem 3.1]{offner2013spectral}.
 Let $m+l\neq 0$, then
\begin{equation*}
(A_{m,l};A_{m,l})_{\L^2(\T,h)}\cdot|\tilde{u}_{m,l}|=(u;A_{m,l})_{\L^2(\T,h)} 
\stackrel{\eqref{Eigenwertgleichung}}{=}\left(u;\frac{D A_{m,l}}{\lambda_{m,l}}\right)_{\L^2(\T,h)}
\end{equation*}
holds. 
We use the fact that the differential operator $D$ is potentially self-adjoint in $\mathring{\T}$ and that the boundary is a  set of measure zero. 
Due to the operator $D$ to be potentially self-adjoint, we find a positive  $C^2$-function $g$ so that $gD$ is self-adjoint, $\frac{1}{g}$ is well defined and symmetric.
 So we get  
\begin{align*}
\left(u;\frac{D A_{m,l}}{\lambda_{m,l}}\right)_{\L^2(\T,h)} &= \left(u;\frac{(gD) A_{m,l}}{g\cdot \lambda_{m,l}}\right)_{\L^2(\T,h)}= \frac{1}{\lambda_{m,l}}\left(\frac{1}{g}(gD)u;A_{m,l} \right)_{\L^2(\T,h)} \\
&\stackrel{\text{recursive}}{=}\left( \frac{1}{\lambda_{m,l}}\right)^k \left(D^ku;A_{m,l} \right)_{\L^2(\T,h)} 
\end{align*}
and hence have shown
\begin{equation}\label{Theorem1.1}
 (A_{m,l};A_{m,l})_{\L^2(\T,h)}\cdot|\tilde{u}_{m,l}|=\left( \frac{1}{\lambda_{m,l}}\right)^k \left(D^ku;A_{m,l} \right)_{\L^2(\T,h)}. 
\end{equation}

By this and
\begin{equation*}
 u(x,y)=\sum\limits_{m,l\in \N_0} \tilde{u}_{m,l} A_{m,l}(x,y)
\end{equation*}
is satisfied 
for every point $(x,y) \in \T$, see \cite[equation (18)]{offner2013spectral}, we get
\begin{align*} 
&|u(x,y)-u_N^{\sigma}(x,y)|= \Bigg|\sum\limits_{\substack{0\leq l+m\leq N\\l,m \in \N_0} }\left(1-\sigma\left(\frac{l+m}{N}\right)\right)\tilde{u}_{m,l}A_{m,l}(x,y)
+\sum\limits_{\substack{l+m> N\\l,m \in \N_0} } \tilde{u}_{m,l}A_{m,l}(x,y) \Bigg|\\
&= 
\Bigg| \int\limits_{\T}h(x_1,y_1)\left[ S_N(x,y,x_1,y_1) +R_N(x,y,x_1,y_1)\right] D^ku(x_1,y_1)\; \d x_1 \;\d y_1 \Bigg|,
\end{align*}
for the function $u$, 
where $S_N(x,y,x_1,y_1)$ and $R_N(x,y,x_1,y_1)$ are defined by 
\begin{align*}
S_N(x,y,x_1,y_1)&=\sum\limits_{\substack{1\leq l+m\leq N\\l,m \in \N_0} }\left(1-\sigma\left(\frac{l+m}{N}\right)\right) \frac{A_{m,l}(x_1,y_1)A_{m,l}(x,y)}{||A_{m,l}||^2_{\L^2(\T,h)} \lambda^{k}_{m,l}},\\
R_N(x,y,x_1,y_1)&=\sum\limits_{\substack{l+m> N\\l,m \in \N_0}}\frac{A_{m,l}(x_1,y_1)A_{m,l}(x,y)}{||A_{m,l}||^2_{\L^2(\T,h)} \lambda_{m,l}}.
\end{align*}
We apply the Cauchy-Schwarz inequality to estimate the integral. Therefore we have to estimate the weighted $\L^2(\T)$-norms of  $S_N$ and $R_N$ for all $(x,y)$. 
For the $R_N$ we extract the results from the proof of  Theorem 3.1 from \cite{offner2013spectral}. We summarize the results in Table \ref{Tabellebesser}.
\begin{table}[h]
\begin{center}
\begin{tabular}[]{|c|c|c|} \hline
area & $\max\{\binom{l+p}{l}, \binom{l+\beta-1}{l}  \}$ & $ ||R_N(x,y,\cdot,\cdot)||_{\L^2(\T,h)}<$ \\ \hline \hline
$\mathring{\T}$ & & $ C_{R_1} N^{-2k+\frac{7}{4}} $\\
$[0,y]$ & $\binom{l+p}{l}$ & $ C_{R_{2(1)}}N^{-2k+\frac{3}{2}\alpha+p+\frac{1}{2} } $ \\ 
$[0,y]$ & $\binom{l+\beta-1}{l} $ & $ C_{R_{2(2)}}N^{-2k+\frac{3}{2}\alpha+\beta-\frac{1}{2} }$ \\
$[x,0]$  &  & $ C_{R_3}N^{-2k+\frac{3}{4}+\beta}$\\
$[x,1-x]$ &  & $C_{R_4}N^{-2k+\frac{7}{4}+p}$\\
$(1,0)$&  & $C_{R_5}N^{-2k-\alpha+\gamma+\frac{3}{2}}$\\
\hline
\end{tabular}\caption{$C_{R_1},\ldots,C_{R_5} \in \R^+$ are constants.}\label{Tabellebesser}
\end{center}
\end{table}
Now we are able to estimate the $S_N$. We follow the proof of \cite[Theorem 3.1]{offner2013spectral} and start with some inner point $(x,y)\in \mathring{\T}$.
It is 
\begin{align*}
 &||S_N(x,y,\cdot,\cdot)||^2_{\L^2(\T,h)}= \sum\limits_{\substack{1\leq l+m\leq N\\l,m \in \N_0} }\left(1-\sigma\left(\frac{l+m}{N}\right)\right) ^2 
 \frac{A^2_{m,l}(x,y)}{||A_{m,l}||^2_{\L^2(\T,h)} \lambda^{2k}_{m,l}}\\
&\stackrel{\substack{\eqref{1druchdieNormAPK} \& \\ \eqref{BetragAPKimInneren}}}{\leq} \sum\limits_{\substack{1\leq l+m\leq N\\l,m \in \N_0} }\left(1-\sigma\left(\frac{l+m}{N}\right)\right) ^2 
\frac{\kappa^2_{m,l} 4 (m+l+\gamma)^2 \tilde{E}^2(x,y)}{\kappa_{l,m}^2(2l+\beta+p)^\frac{1}{2}(m+l)^{2k} (m+l+\gamma)^{2k} }\\
&\leq 4 \tilde{E}^2(x,y) \sum\limits_{\substack{1\leq l+m\leq N\\l,m \in \N_0} } \left(1-\sigma\left(\frac{l+m}{N}\right)\right) ^2  \frac{1}{(m+l)^{2k+\frac{1}{2}}(m+l+\gamma)^{2k-2} }\\
\end{align*}
Using the identity 
\begin{equation*}
 \sum\limits_{\substack{0<l+m\leq N\\l,m \in \N_0}} \left(\frac{1}{m+l} \right)^k=\sum\limits_{i=1}^N \frac{i+1}{i^k}
\end{equation*}
as well as elementary estimates lead to 
\begin{align*}
||S_N(x,y,\cdot,\cdot)||^2_{\L^2(\T,h)}&< 8\tilde{E}(x,y)^2 \sum\limits_{i=1}^N\left(1-\sigma\left(\frac{i}{N}\right)\right) ^2 \frac{1}{i^{4k-\frac{5}{2}}} \frac{N^{4k-\frac{5}{2}}}{N^{4k-\frac{5}{2}}}\\
&=8\tilde{E}(x,y)^2 N^{-4k+\frac{7}{2}}\left( \frac{1}{N} \sum\limits_{ i=1}^N\left(1-\sigma\left(\frac{i}{N}\right)\right) ^2 \left(\frac{i}{N}\right)^{-4k+\frac{5}{2}} \right).
 \end{align*}
If $N\to \infty $, the last factor conforms with the integral
\begin{equation*}
 \int\limits_{0}^1(1-\sigma(\tau))^2 \tau^{\frac{5}{2}-4k}\d \tau.
\end{equation*}
The integral is bounded by 
the given conditions. To verify this fact, we calculate    the Taylor series of $\sigma$ in $0$. It is 
\begin{equation*}
 \sigma(\tau)=\sum\limits_{j=0}^{2k-1}\frac{1}{j!}\sigma^{(j)}(0)\tau^j+\ol(\tau^{2k-1})  \quad \forall \tau \in [0,\varepsilon).
\end{equation*}
In connection with  $\sigma(0)=1$ and $\sigma^{(j)}(0)=0$ for all $j=1,2,\ldots,2k-2$ the term reduces to 
\begin{equation*}
 \sigma(\tau)=1+\frac{1}{(2k-1)!}\sigma^{(2k-1)}(0)\tau^{2k-1}+ \ol(\tau^{2k-1}) 
\end{equation*}
for all $ \tau \in [0,\varepsilon)$. Using this in the integral, we get
\begin{align*}
 \int\limits_{0}^\varepsilon \left(\frac{\sigma^{(2k-1)}(0)}{(2k-1)!}\tau^{2k-1}+\ol(\tau^{2k-1}) \right)^2\tau^{-4k+\frac{5}{2}} \d \tau 
 + \int\limits_\varepsilon^1  (1-\sigma(\tau))^2 \tau^{\frac{5}{2}-4k}\d \tau <C^2_{I_1},
\end{align*}
because every part is bounded.  
Hence it applies 
\begin{equation*}
 ||S_N(x,y,\cdot,\cdot)||_{\L^2(\T,h)} <\sqrt{8} \tilde{E}(x,y)N^{-2k+\frac{7}{4}}C_{I_1}=C_{S_1}N^{-2k+\frac{7}{4}},
\end{equation*}
and with the Cauchy-Schwarz inequality we show  in total 
\begin{align*}
 |u(x,y)-u_N^{\sigma}(x,y)| \leq& ||D^ku||_{\L^2(\T,h)} \left(||R_N(x,y,\cdot,\cdot)||_{\L^2(\T,h)} +||S_N(x,y,\cdot,\cdot)||_{\L^2(\T,h)}  \right)  \\ 
 \leq& \underbrace{(||D^ku||_{\L^2(\T,h)} )(C_{R_1}+C_{S_1})}_{K_1} N^{-2k+\frac{7}{4}}.
\end{align*}
The computations for the edges are analogously obtained.
We present the results in the following table.
\begin{center}
\begin{tabular}[]{|c|c|c|} \hline
edge & $\max\{\binom{l+p}{l}, \binom{l+\beta-1}{l}  \}$ & $|u(x,y)-u_N^\sigma(x,y)|<$ \\ \hline \hline
$[0,y]$ & $\binom{l+p}{l}$ & $ \frac{K_2}{N^{2k-\frac{3}{2}\alpha-p-\frac{1}{2}}},$  \\ 
$[0,y]$ & $\binom{l+\beta-1}{l} $ & $  \frac{K_3 }{N^{2k-\frac{3}{2}\alpha-\beta+\frac{1}{2}}}$, \\
$[x,0]$  &  & $  \frac{K_4}{N^{2k-\frac{3}{4}-\beta}}$\\
$[x,1-x]$ &  & $  \frac{K_5}{N^{2k-\frac{7}{4}-p}}$\\
\hline
\end{tabular}
\end{center}

For the last point $(1,0)$ we calculate it directly.
\begin{align*}
&||S_N(1,0,\cdot,\cdot)||^2_{\L^2(\T,h)} 
\stackrel{}{=}  \sum\limits_{m=1 }^N\left(1-\sigma\left(\frac{m}{N}\right)\right) ^2\frac{A^2_{m,0}(1,0)}{||A_{m,0}||^2_{\L^2(\T,h)} \lambda^{2k}_{m,0}}\\
&\stackrel{\eqref{1druchdieNormAPK}}{<} 4 \left( \frac{(\beta)_p \gamma^{\alpha-1}}{(1)_p} \right) \sum\limits_{m=1}^N 
\left( 1-\sigma \left(\frac{m}{N} \right) \right)^2 \frac{\binom{m+\gamma-\alpha}{m}^2 }{m^{2k}(m+\gamma)^{2k-2} }\\
&<C_5 \sum\limits_{m=1}^N 
\left( 1-\sigma \left(\frac{m}{N} \right) \right)^2 \frac{ 1}{m^{2k}(m+\gamma)^{2k-2-2\gamma+2\alpha} } \frac{N^{4k-2-2\gamma+2\alpha}}{N^{4k-2-2\gamma+2\alpha}}\\
&<C_{S_{5}}N^{-4k+3 + 2\gamma-2\alpha}\left( \frac{1}{N}\sum\limits_{m=1}^N \left(1-\sigma\left(\frac{m}{N}\right) \right)^2 \left( \frac{m}{N}\right)^{{-4k+2+2\gamma-2\alpha}} \right),
\end{align*}
and after all it ensues 
\begin{equation*}
 |u(1,0)-u_N^{\sigma}(1,0)|<K_6 N^{-2k-\alpha+\gamma+\frac{3}{2}}.
\end{equation*}
\end{proof}

To verify Theorem \ref{Theorem1} we test the approximation speed. 
On the triangle $\T$, we approximate the function $f(x,y)= \sin\left( \pi 
(x+y) \right)$ by the filtered and unfiltered truncated series expansion with 
respect to the APK polynomials for parameters $(\alpha, \beta,\gamma)=(1,1,2)$ 
and $(2,2,5)$.
For the filters, we choose exponential filters \eqref{Exponentialfilter} of 
order $2$ and filter strength $0.001$ and of
order $4$ and strength $0.00001$. Furthermore, we applied the classical 
cosine-filter
\begin{equation*}
 \sigma(\eta)=0.5\left( 1+ \cos\l( \pi \eta\r) \right)
\end{equation*}
of second order.
In Figure \ref{error} the maximal errors of these truncated series expansions 
are illustrated with respect to the polynomial degree . 
$A$ and $E$ are the unfiltered APK-expansions, $B, F$ are the filtered 
APK-expansions, where the exponential filter of order $2$ was used, for $C,G$ 
the cosine-filter was used and for $D,H$ the exponential filter of order $4$ 
was used.
Therefore, $A$ is comparable to $ E$, $B$ to $F$, $C$ to $G$ and $D $ to $H$.
All calculations were done in Mathematica and by its integration 
subroutine in order to calculate the coefficients $\hat{u}_{m,l}$.

\begin{figure}[!htb]
\centering
\includegraphics[width=0.5\textwidth]{%
  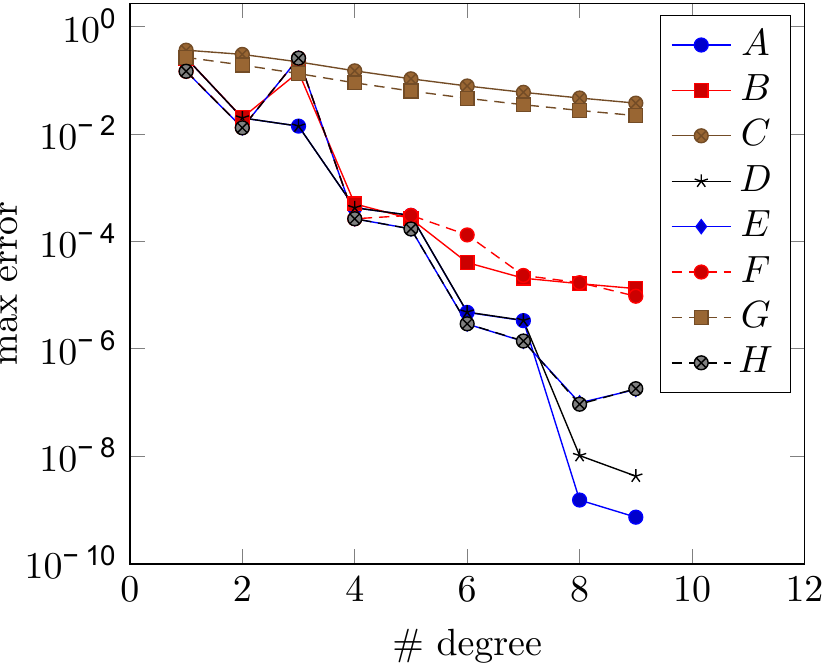}
  \caption{$A$-$D$ parameters $(\alpha, \beta,\gamma)=(1,1,2)$ and $E$-$H$ parameters $(2,2,5)$ 
}
  \label{error}
\end{figure}

As Figure \ref{error} illustrates, the highest rate of convergence is obtained 
by the unfiltered APK expansions. On the other hand, the APK expansions 
filtered by the cosine-filter (C,G) and the exponential filter of second order 
(B,F) show a notably slower rater of convergence.
Yet, by increasing the order of the filter and applying the exponential filter 
of fourth order to the APK expansion (D,H), the rate of convergence 
significantly increases.
Thus, for higher filter orders we get a stronger decrease in the $\max$-error. 
In Table \ref{Fehlerposition} the position and the value of the maximum-errors 
are plotted for $C$ and $G$.
\begin{table}[!htb]
\begin{center}
 \begin{tabular}{c|c|c||c|c }
  $N$ & $(x,y)$ &  $\max$- error ($C$-case )  &    $(x,y)$ &  $\max$- error $(G)$-case  \\ \hline \hline
 1 & (0.331, 0.169) &  0.36338   & (0.331, 0179)  & 0.266869 \\
 2 & (0.269, 0.273) & 0.303294   & (0.198, 0.345)  & 0.192026 \\
 3 & (0.118, 0.435) & 0.219701 & (0.129, 0.442) & 0.132915 \\
 4 & (0.162, 0.391) & 0.149895 &  (0.219, 0.353 ) & 0.089492 \\
 5 & (0.261, 0.293) & 0.106279 & ( 0.284, 0.289 ) & 0.0629229  \\
 6 & (0.263, 0.291) & 0.0781872 & (0.258, 0.316 ) & 0.0460719\\
 7 & (0.197, 0.357) & 0.0595273 & (0.196, 0.378) & 0.0273561\\ 
 8 & (0.220, 0.334) & 0.0466554 & (0.137, 0.437 ) & 0.0219375  \\
\end{tabular} 
\caption{$ C$ and $ G$}\label{Fehlerposition}
\end{center}
\end{table}
For the others it looks analogous.\\
The maximum error lies in  $\mathring \T$ and the error decreases with 
$\frac{1}{N^{1.25}}$. 
For $C$ $K1$ has to be around $0.867428$ and for $G$ $K1$ is $0.52477$.
In Table \eqref{Konstants}, we numerically determined the constants $K1,\; 
K3,\; K4,\; K5$ and $K6$ for the case $D$. 

\begin{table}[!htb]
\begin{center}
 \begin{tabular}{c|c|c|c|c }
  $K1$ & $K3$ & $K4$  &  $K5$ &  $K6$  \\ \hline \hline
  $0.496$ & $0.741 $  & $ 0.975$ &  $ 0.496$ &   $ 0.428$ 
\end{tabular} 
\caption{$ K1-K6$}\label{Konstants}
\end{center}
\end{table}
All of this verifies our result.

\subsection{The Spectral Viscosity Method}\label{SVM}

Since spectral methods are known to lack sufficient dissipation, Tadmor \cite{tadmor1989convergence} proposed the Spectral Viscosity or Super-Spectral Viscosity Method (SV Method). 
The main idea of the SV Method is to add a small viscosity term to the conservation law \eqref{hypGleichung}. We show analogously to 
\cite{ma1998chebyshev, ortleb2011diskontinuierliches} that by choosing a viscosity term which depends on the differential operator of the orthogonal basis, the 
SV Method can be seen as a spectral method with modal filtering. \\ 
For the SD scheme based on the APK expansions, we propose to consider high order operators of the form $(-D)^p$on the reference element $\T$.
Let $I_N$ be an index set and  $\phi_k=A_{m,l}$ the APK polynomials in lexicographic order $k$, see Section \ref{SD Method}.
\begin{thm}\label{Spekralersatz}
 Let $N\in \N$ and $\{\phi_k | k\in I_N \}$ the APK polynomials on the triangle $\T$. The polynomials solve 
the eigenvalue problem 
 \begin{equation}\label{Eigenwertgleichungfilter}
  -D\phi_k=-\lambda_{k} \phi_k, 
  \end{equation}
  where $D$ is the differential operator of the polynomials. 
$P_\phi$ denotes the projection on the space of $\{ \phi_k\}$.  To solve the viscosity equation
\begin{equation}\label{viskoseAPKreihemitfilter}
 \frac{\partial }{\partial t} u_N(\boldsymbol{x},t) +\nabla_{\boldsymbol{x}} \cdot P_\phi f(u_N({\boldsymbol{x}},t))=\varepsilon_N(-1)^{p+1} (-D)^p u_N({\boldsymbol{x}},t)
\end{equation}
by a splitting method is equivalent to multiply the coefficients $\hat{u}_{k}$ with the function 

\begin{equation*}
 \sigma(k)=e^{-\varepsilon_N \Delta t \lambda_{k}^p}
\end{equation*}
in every update step of the equation 
\begin{equation*}
 \frac{\partial }{\partial t} u_N({\boldsymbol{x}},t) +\nabla_{\boldsymbol{x}} \cdot P_{\phi} f(u_N({\boldsymbol{x}},t))=0.
\end{equation*}
\end{thm}

\begin{proof}
We solve the equation \eqref{viskoseAPKreihemitfilter} by a splitting method in two steps
 \begin{equation}\label{Splitting1}
  \frac{\partial }{\partial t} u_N({\boldsymbol{x}},t)= \varepsilon_N(-1)^{p+1} (-D)^p u_N({\boldsymbol{x}},t)
 \end{equation}
and 
\begin{equation}\label{Splitting2}
 \frac{\partial }{\partial t} u_N({\boldsymbol{x}},t) + \nabla_{\boldsymbol{x}} \cdot P_\phi f(u_N({\boldsymbol{x}},t)=0
\end{equation}
With $u_N({\boldsymbol{x}},t)=\sum\limits_{k\in I_N} \hat{u}_k \phi_k({\boldsymbol{x}})$ equation \eqref{Splitting1} implies 
\begin{equation*}
  \sum\limits_{k\in I_N} \frac{\partial \hat{u}_k }{\partial t} \phi_k({\boldsymbol{x}})= \sum\limits_{k\in I_N} \varepsilon_N(-1)^{p+1} \hat{u_k} (-D)^p
  \phi_k({\boldsymbol{x}})=\sum\limits_{k\in I_N} -\varepsilon_N \hat{u}_k(t) 
  \lambda_k^p\phi_{k} 
  ({\boldsymbol{x}}),
\end{equation*}
where we applied the eigenvalue equation \eqref{Eigenwertgleichungfilter} in the last step. 
By comparing 
the coefficients we have to solve the ordinary differential equations 
\begin{equation*}
 \frac{\partial \hat{u}_k(t)}{\partial t}= -\varepsilon_N \hat{u_k}(t) \lambda_k^p, \quad \forall k\in I_N.
\end{equation*}
The solution is  $\hat{u}_k(t)=c e^{-\varepsilon_N \lambda_k^p t}$, $c\in \R$. 
With $\Delta t:=t^{n+1}-t^n$ and the requirement $\hat{u}_k(t^{n+1})=\hat{u}_k(t^n)$  for $\Delta t=0$ follows
\begin{equation*}
 \hat{u}_k(t^{n+1})=e^{-\varepsilon_N\lambda_k^p(\Delta t+t^n)}=\underbrace{e^{(-\varepsilon_N \lambda_k^p \Delta t)}}_{=:\sigma(k)} \hat{u}_k(t^n).
\end{equation*}
\end{proof}
In order of us to speak of a modal filter for $\sigma(k)=\sigma((l,m))=e^{-\varepsilon_N(l+m)^p(l+m+\gamma)^p \Delta t}$, we have to multiply the exponenet with  $N^{2p}$, meaning
\begin{equation}\label{Filterapproximiert}
 \sigma \left( \frac{l+m}{N} \right)= e^{-\varepsilon_N N^{2p} \left( \frac{l+m}{N}  \right)^{p} \left( \frac{l+m+\gamma}{N}  \right)^{p} \Delta t} 
                             \approx e^{-\varepsilon_N N^{2p} \Delta t \left( \frac{l+m}{N}  \right)^{2p} }.
\end{equation}
Hence  $\sigma:[0,1]\to [0,1]$ can be seen as an exponential filter of order $2p$ with filter strength 
$\alpha_i:= -\varepsilon_N N^{2p} \Delta t$. 
We consider \eqref{Filterapproximiert} in detail and realize that our modal filter depends on the parameter $\gamma$ of the APK polynomials, especially
if the polynomials have  minor degree. So for different families of APK polynomials we get various specific modal filters. \\  
We can prove similar results to  Theorem \ref{Spekralersatz} for every orthogonal basis $\{\phi_k \}$, if the $\phi_k$ fulfil a comparable eigenvalue equation.\\
For the  transfer of the SV Method to our Spectral Difference Method we follow the steps according 
to \cite{wirz2012spektrale}. The transformation $T_i$ from $\tau_i$ to $\T$ has no effect on the flux function, so that the SD update scheme of a scalar hyperbolic equation in the 
cell $\tau_i$ with viscosity term and  the solution points $\xi_j\in \T$ reads like 
\begin{equation}\label{SDaktualisierungschema}
\begin{aligned}
 &\frac{\partial }{\partial t } u_N(T_i^{-1}({\boldsymbol{\xi}}_j,t))+\nabla_{\boldsymbol{\xi}} \cdot \tilde{P}_N (J_{T_i})^T \tilde{F}(u_N(T_i^{-1}(\boldsymbol{\xi}_j),t))\\
 =&\varepsilon_N(-1)^{p+1}(-D)^p(u_N(T_i^{-1}(\boldsymbol{\xi}_j),t)),
\end{aligned}
\end{equation}
where $\tilde{F}$ is the flux function, $\tilde{P}$ is the projection and $\nabla_{\boldsymbol{\xi}}$ is the nabla operator on $\T$.
The same approach as 
in Theorem  \ref{Spekralersatz} yields  the equivalent of \eqref{SDaktualisierungschema} and the SD Method with the modal filter \eqref{Filterapproximiert}.\\
  Hence, the matrix representation (\ref{schemeMatrix2}) for the SD update scheme in cell $\tau_i$ becomes
  \begin{align}
    \frac{\d \mathbf{u}^s}{\d t}(t) = 
    & - ( \xi_x D_\xi + \xi_y D_\eta ) \bigl( \sigma(k) \hat{F}_{k,1}(t) \bigr)_k \nonumber \\
    & - ( \eta_x D_\xi + \eta_y D_\eta ) \bigl( \sigma(k) \hat{F}_{k,2}(t) \bigr)_k \nonumber \\
    = & - ( \xi_x D_\xi + \xi_y D_\eta ) M_\sigma \hat{\mathbf{F}}_1(t) \nonumber \\
    & -( \eta_x D_\xi + \eta_y D_\eta ) M_\sigma \hat{\mathbf{F}}_2(t) \nonumber \\
    = & - ( \xi_x D_\xi + \xi_y D_\eta ) M_\sigma \mathcal{V}^{-1} \mathbf{F}_1(t) \\
    & - ( \eta_x D_\xi + \eta_y D_\eta ) M_\sigma \mathcal{V}^{-1} \mathbf{F}_2(t) \nonumber
  \end{align}
  with viscosity, where $M_\sigma := \diag{\sigma(1),\dots,\sigma(N_F)}$ is a diagonal 
matrix. 
So we have to choose the size of  $\varepsilon_N$, thereby we stick on 
the analysis of \cite{gottlieb2001spectral}. The authors suggested in the Fourier case a viscosity strength $\varepsilon_N$, which depends on the approximation order as 
\begin{equation*}
 \varepsilon_N\sim \frac{C_P}{N^{2p-1}},
\end{equation*}
where the constant $C_P$ may be chosen by 
\begin{equation*}
 C_p\leq \sum\limits_{k=1}^p||\partial_u^k (J_{T_i})^T \tilde{F}(u) ||_{\L^\infty}= ||(J_{T_i})^T||_{\infty} \sum\limits_{k=1}^p||\partial_u^k  \tilde{F}(u)
  ||_{\L^\infty}.
\end{equation*}
Here, we ignore the dependence on $u$ and approximate $||(J_{T_i})^T||_{\infty} \sim \frac{1}{h_i}$ by a measure of length $h_i$ of the triangle $\tau_i$. Finally we arrive 
at 
\begin{equation}\label{viscosityStrength}
 \varepsilon_N^i:=\frac{c}{h_iN^{2p-1}}< \frac{\sum\limits_{k=1}^p||\partial_u^k  \tilde{F}(u)
  ||_{\L^\infty}}{h_iN^{2p-1}},
\end{equation}
where $c\in \R$ is a constant.
We have  to select the constant $c$ and the filter order $2p$ in our numerical tests. 
\begin{rem}
The numerical tests 
from \cite{meister2013extended, wirz2012spektrale} imply 
that we have to increase $c$ if we enlarge the filter order $2p$.\\
Furthermore, it is a known fact that 
the application of a modal filter in the global domain will destroy the order of convergence. 
Shock indicators 
are used to detect corrupted cells. There 
are different approaches to detect jumps in a cell, see \cite{offner2013detecting, wirz2014detecting} for example. 
Here we use 
the shock indicator of \cite{barter2007shock}. The indicator compares the 
higher 
coefficients with the lower ones to detect oscillations. 
\end{rem}

  In general, the time step $\Delta t$ in the explicit time-stepping 
  scheme\footnote{We already mentioned, that we use the 4-th order 
  low storage Runge-Kutta scheme defined by Carpenter and Kennedy
  (see \cite{carpenter1994fourth}) in our implementation for Section
  \ref{Numerical Test}.} for the resulting ODEs (in $t$) at each solution
  point ${\bf x}_j$, which also appears in the filter strength 
  $\alpha_i = -\varepsilon_N N^{2p} \Delta t$ in (\ref{Filterapproximiert}), 
  can't be chosen arbitrarily. For a scalar conservation law $u_t(x,t) + \lambda u_x(x,t) = 0$ 
  in one space dimension and a numerical scheme on a uniform grid 
  with length $h$ one should choose $\Delta t$ small enough, so that 
  a wave with propagation speed $\lambda$ can't travel more than $h$ in 
  one time step. I.e. the \textit{Courant-Friedrichs-Levy condition} $\lambda \Delta t \leq h$ has to hold. 
  The maximum 
  \begin{align*}
    \underset{\lambda}{\max} \; \frac{\lambda \Delta t}{h} =: \mathrm{CFL}
  \end{align*}
  is called \textit{Courant number} and the following example gives a first intuition how it can determined on a triangulation.
  \begin{ex}\label{bspl}
    We observe the two dimensional advection equation 
    \begin{align*}
      u_t(x,y,t) + u_x(x,y,t) + u_y(x,y,t) = 0.
    \end{align*}
    Here the propagation speed in $(1,1)$ direction is $\sqrt{2}$, 
    which leads to the Courant number $\mathrm{CFL} = \frac{\sqrt{2} \Delta t}{h}$, 
    in which $h$ is now a measure of length of the observed triangle. 
  \end{ex} 
  As we can see in Section \ref{Stability}.2 
  the SD Method has stability problems 
  for all parameters $\alpha,\beta$ and $\gamma$, 
  which are increasing for rising order $N$. 
  Hence, we can't even compute a sufficiently numerical Courant number and will choose the time step
  \begin{align}\label{timestep}
    \Delta t := \frac{C_{fix}}{(N+1)^2} \cdot \frac{h}{\lambda_{max}}
  \end{align}
  with a fixed value $C_{fix}$ and maximal propagation speed $\lambda_{max}$.
  Therefore, the time step adapts an increasing instability for rising $N$, 
  the "geometry" of the triangulation by $h$ and by $C_{fix} = \frac{1}{2}$, which will be our choice,
  (\ref{timestep}) coincides with the Courant number in example \ref{bspl} for $N=0$.

\section{Stability of the Spectral Difference Method}\label{Stability}

By von Neumann analysis in matrix form for a scalar advection equation in two spatial dimensions with periodic initial condition, one can observe 
the linear stability of the semi-discretization. This was already done by Abeele, Lacor and Wang for the 
classical 
SD Method in \cite{van2008stability}, 
by Huynh for the FR Method in \cite{huynh2007flux} and many other authors. Here we observe the SD Method for APK polynomials
on a triangular grid, which is generated by a pattern like
\begin{figure}[!htb]
  \centering
  \begin{subfigure}[b]{0.7\textwidth}
    \includegraphics[width=\textwidth]{%
    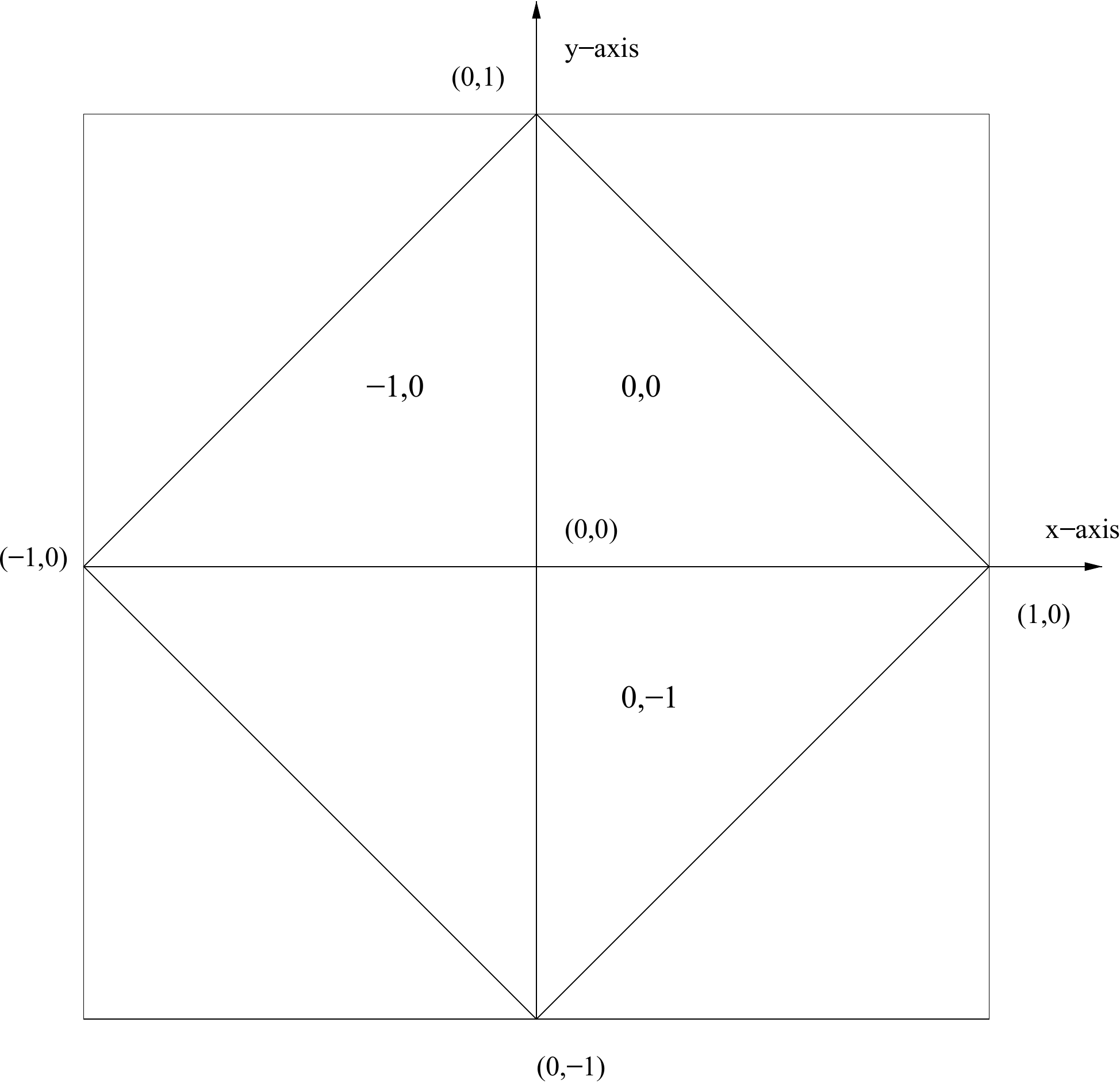} 
  \end{subfigure}
  \caption{Generating pattern}
  \label{grid}
\end{figure}
in Figure \ref{grid} for a scalar advection equation   
\begin{align*}
  u_t+ \cos(\psi) u_x + \sin(\psi) u_y = 0
\end{align*}
with $\psi \in [0,\frac{\pi}{2}]$ and the periodic initial condition
\begin{align*}
  u_{init}(x,y) = e^{I(w_x x + w_y y)},
\end{align*}
where $w_x,w_y\in [-\pi,\pi]$ are the so called \textit{wave numbers} in $x$- and $y$-direction. The exact solution is then given by
\begin{align}\label{exactsolution}
  u_{exact}(x,y,t) = e^{ I( w_x x + w_y y - [ w_x \cos \psi + w_y \sin \psi ]t ) }
\end{align}
for $t \geq 0$. We will concentrate on the semi discretization at $t=0$ and wish to obtain a relation in form of a linear ODE system
\begin{align}\label{ODE}
  \frac{\d{\bf u}^s}{\d t} = S \bf{u}^s
\end{align}
of first order and with constant coefficients in the solution points. 
Similar to the scalar case $\frac{\d u}{\d t} = \lambda u$, 
where the solution is given by $u(t) = ce^{\lambda t}$ for $c \in \mathbb{R}$, 
the eigenvalues of the matrix $S$ take the place of $\lambda$. Hence the asymptotic 
behavior of the solution of (\ref{ODE}) depends on the real parts of the eigenvalues: 
If all eigenvalues of $S$ lie in the left half of the complex plane, the solution is 
bounded, which coincides with the behavior of the exact solution $u_{exact}$. In this 
case we call the semi discretization to be \textit{stable}. On the other hand, if any 
eigenvalue lies in the right half of the complex plane, the solution of (\ref{ODE}) will 
blow up for $t \to \infty$ and so might the scheme. This will happen unless the \textit{instability}
is mild, i.e. $\underset{\lambda}{\max} \: \mathrm{Re}(\lambda)$ is small, and the time steps 
adapt to
this and are sufficiently small. In $5.1$ we will first derive a matrix representation 
like (\ref{ODE}) for the SD scheme at $t=0$ and for the test case above. In $5.2$ we will
 analyze the stability of the unfiltered SD scheme from Section \ref{SD Method} for 
certain ranges for $\alpha, \beta$ and $\gamma$. By observation of the numerical test cases, 
one can see that some stabilization is necessary and we did so in Section \ref{Filter} by the
SV Method in terms of the SD Method with modal filtering. Following this, in $5.3$ we will do 
a similar analysis to $5.2$ for the SD Method with modal filtering. One can see how the modal 
filtering leads to much better stability properties and how they depend on certain parameters 
like the order $p$ and $\alpha,\beta,\gamma$ from the APK polynomials.

\subsection{Matrix representation}

We already stated in Section \ref{SD Method} that the universal update scheme (\ref{update}) at the cell $\tau_{i,j}$ reads
\begin{align}\label{matrix1}
  \frac{ \d {\bf u}_{i,j}^s(t)}{ \d t} = - (\xi_x D_\xi + \xi_y D_\eta)\mathcal{V}^{-1} {\bf 
F}_1(t) - (\eta_x D_\xi + \eta_y D_\eta)\mathcal{V}^{-1} {\bf F}_2(t),
\end{align}
when evaluated at the solution points. If we look at the cell $\tau_{0,0}$ and denote 
$\frac{ \d {\bf u}_{i,j}^s(t=0)}{ \d t}$ for sake
of simplicity by $\frac{ \d {\bf u}^s}{ \d t}$ equation (\ref{matrix1}) becomes
\begin{align}\label{matrix2}
  \frac{ \d {\bf u}^s}{ \d t} = - ( D_\xi \mathcal{V}^{-1} {\bf F}_1 + D_\eta 
\mathcal{V}^{-1} {\bf F}_2),
\end{align}
since $\tau_{0,0}$ and its neighbors are parametrized by the orientation-preserving affine transformations
\begin{displaymath}
  T_{0,0}(x,y)=\begin{pmatrix} x \\ y \end{pmatrix}, 
  T_{-1,0}(x,y)=\begin{pmatrix} y \\ -x \end{pmatrix},
  T_{0,-1}(x,y)=\begin{pmatrix} -y \\ x \end{pmatrix},
  T_{0,1}(x,y)=\begin{pmatrix} 1-x \\ 1-y \end{pmatrix}.
\end{displaymath}
Note that ${\bf F}_1,{\bf F}_2$ are given by (\ref{numerical flux}), where the upwind flux is used and the flow direction was restricted to $\psi \in [0,\frac{\pi}{2}]$. 
Consequently ${\bf F}_1,{\bf F}_2$ are independent of the values from $\tau_{0,1}$ and so (\ref{matrix2}) now reads
\begin{align*}
  \frac{\d {\bf u}^s}{\d t} = & - \cos(\psi)D_\xi \mathcal{V}^{-1} \bigl[ M_1^{-1,0} 
{\bf u}^f_{-1,0} + M_1^{0,0} {\bf u}^f_{0,0} + M_1^{0,-1} {\bf u}^f_{0,-1} \bigr] \\
  & - \sin(\psi)D_\eta \mathcal{V}^{-1} \bigl[ M_2^{-1,0} {\bf u}^f_{-1,0} + M_2^{0,0} {\bf u}^f_{0,0} + M_2^{0,-1} {\bf u}^f_{0,-1} \bigr].
\end{align*}
Here the $M_{i}^{m,n}$ for the second order SD Method, and so $K_F=\frac{(2+1)(2+2)}{2}=6$, are
\begin{align*}
  M_1^{-1,0}=e_1^Te_1+e_2^Te_2+e_3^Te_3, \; M_1^{0,0}=e_5^Te_5+e_6^Te_6, \; M_1^{0,-1}=e_4^Te_4 , \\
  M_2^{-1,0}=e_2^Te_2, \; M_2^{0,0}=e_3^Te_3+e_5^Te_5, \; M_2^{0,-1}=e_1^Te_1+e_4^Te_4+e_6^Te_6. 
\end{align*}
For the exact computation see \cite[chapter $3.3.1$]{wirz2012spektrale}  and the $M_{i}^{m,n}$ for $N=3,4,5$ can be found in the appendix. 
Next we can make use of the periodicity of the initial condition to get rid of ${\bf u}^f_{-1,0}$ and ${\bf u}^f_{0,-1}$ by substituting 
them by terms $T_{-1,0}{\bf u}^f_{0,0}$ and $T_{0,-1}{\bf u}^f_{0,0}$. Therefore, note that ${\bf u}^f_{-1,0},{\bf u}^f_{0,0}$ and ${\bf u}^f_{0,-1}$
are given by the initial condition. Hence
\begin{align*}
  {\bf u}^f_{0,0} = & (u_{init}(T^{-1}_{0,0}(\xi_k,\eta_k)))_k = (u_{init}(\xi_k,\eta_k))_k \\
  = & (e^{I(w_x\xi_k+w_y\eta_k)})_k, \\
  {\bf u}^f_{-1,0} = & (u_{init}(T^{-1}_{-1,0}(\xi_k,\eta_k)))_k = (u_{init}(-\eta_k,\xi_k))_k \\
  = & (e^{I(-w_x\eta_k+w_y\xi_k)})_k = ( e^{I(-w_x(\xi_k+\eta_k)+w_y(\xi_k-\eta_k))} \cdot e^{I(w_x\xi_k+w_y\eta_k)} )_k \\
  = & T_{-1,0} {\bf u}^f_{0,0}, \\
  {\bf u}^f_{0,-1} = & (u_{init}(T^{-1}_{0,-1}(\xi_k,\eta_k)))_k = (u_{init}(\eta_k,-\xi_k))_k \\
  = & (e^{I(w_x\eta_k-w_y\xi_k)})_k = ( e^{I(w_x(\eta_k-\xi_k)-w_y(\xi_k+\eta_k))} \cdot e^{I(w_x\xi_k+w_y\eta_k)} )_k \\
  = & T_{0,-1} {\bf u}^f_{0,0}
\end{align*} with \begin{align*}
  T_{-1,0} & := \sum_k e^{I(-w_x(\xi_k+\eta_k)+w_y(\xi_k-\eta_k))} e_k^Te_k, \\
  T_{0,-1} & := \sum_k e^{I(w_x(\eta_k-\xi_k)-w_y(\xi_k+\eta_k))} e_k^Te_k,
\end{align*}
and we obtain
\begin{equation}
  \begin{aligned}\label{matrix3}
    \frac{\d {\bf u}^s}{\d t} = - & \Bigl( \cos(\psi)D_\xi \mathcal{V}^{-1} \bigl[ M_1^{-1,0} 
T_{-1,0} + M_1^{0,0} + M_1^{0,-1} T_{0,-1} \bigr] \\
    & + \sin(\psi)D_\eta \mathcal{V}^{-1} \bigl[ M_2^{-1,0} T_{-1,0} + M_2^{0,0} + M_2^{0,-1} T_{0,-1} \bigr] \Bigr) {\bf u}^f_{0,0}.
  \end{aligned}
\end{equation} 
Last, we want to write ${\bf u}^f_{0,0}$ in terms of ${\bf u}^s_{0,0}$ or just ${\bf u}^s$ to obtain a
relation as in (\ref{ODE}). While we reconstructed the flux $F$ by APK polynomials, 
in order to later use their natural filter given by (\ref{Filterapproximiert}), we reconstruct the solution $u$ by Lagrange polynomials. Hence
\begin{align*}
   {\bf u}_{0,0}^f = 
   & \bigl( \sum_{k=1}^{K_s} u_{0,0}({\bf x}^s_k) L_k(T_{0,0}({\bf x}_j^f)) \bigr)_j^{K_F} \\
   = & E_{Lag} {\bf u}^s 
\end{align*} with \begin{align*}
  E_{Lag} := \bigl( L_k({\bf x}_j^f) \bigr)_{j,k} \in \mathbb{R}^{K_F\times K_s}.
\end{align*} If we use that in (\ref{matrix3}) we finally obtain 
\begin{align*}
  \frac{\d {\bf u}^s}{\d t} = - & \Bigl( \cos(\psi)D_\xi \mathcal{V}^{-1} \bigl[ M_1^{-1,0} 
T_{-1,0} + M_1^{0,0} + M_1^{0,-1} T_{0,-1} \bigr] \\
    & + \sin(\psi)D_\eta \mathcal{V}^{-1} \bigl[ M_2^{-1,0} T_{-1,0} + M_2^{0,0} + M_2^{0,-1} T_{0,-1} \bigr] \Bigr) E_{Lag} {\bf u}^s,
\end{align*}
which we want to retain in the following Lemma.
\begin{lem}\label{lemma}
  For a linear advection equation with flow direction $\psi \in [0,\frac{\pi}{2}]$ and periodic initial condition $u_{init}$ on a 
  triangular grid like in Figure \ref{grid}, where the fully upwind is used, 
the universal update scheme (\ref{update}) on $\tau_{0,0}$ reads
  \begin{displaymath}
    \frac{\d {\bf u}^s}{\d t} = S {\bf u}^s
  \end{displaymath}
  with semi-discretization $S$ given by
  \begin{equation}
    \begin{aligned}\label{semidiscr}
      S := - & \Bigl( \cos(\psi)D_\xi \mathcal{V}^{-1} \bigl[ M_1^{-1,0} T_{-1,0} + M_1^{0,0} + M_1^{0,-1} T_{0,-1} \bigr] \\
      & + \sin(\psi)D_\eta \mathcal{V}^{-1} \bigl[ M_2^{-1,0} T_{-1,0} + M_2^{0,0} + M_2^{0,-1} T_{0,-1} \bigr] \Bigr) E_{Lag}.
    \end{aligned}  
  \end{equation}
\end{lem}
We already stated in Section \ref{Filter} that the universal update scheme (\ref{update}) with modal filtering at the cell $\tau_{i,j}$ reads
\begin{equation}
  \begin{aligned}\label{matrix4}
    \frac{\d {\bf u}_{i,j}^s(t)}{\d t} = & - (\xi_x D_\xi + \xi_y D_\eta) M_\sigma 
\mathcal{V}^{-1} {\bf F}_1(t) \\
    & - (\eta_x D_\xi + \eta_y D_\eta) M_\sigma \mathcal{V}^{-1} {\bf F}_2(t),
  \end{aligned}
\end{equation}
when evaluated at the solution points. If we combine this with Lemma \ref{lemma}, we obtain a similar corollary for the the filtered SD Method.
\begin{cor}
  For a linear advection equation with flow direction $\psi \in [0,\frac{\pi}{2}]$ and periodic initial 
  condition $u_{init}$ on a triangular grid like in Figure \ref{grid}, where 
the full upwind is used, the universal update 
  scheme with modal filtering on $\tau_{0,0}$ reads
  \begin{displaymath}
    \frac{\d {\bf u}^s}{\d t} = S_\sigma {\bf u}^s
  \end{displaymath}
  with semi discretization $S_\sigma$ given by
  \begin{equation}
    \begin{aligned}\label{semidiscr2}
      S_\sigma := & - \Bigl( \cos(\psi)D_\xi M_\sigma \mathcal{V}^{-1} \bigl[ M_1^{-1,0} T_{-1,0} + M_1^{0,0} + M_1^{0,-1} T_{0,-1} \bigr] \\
      & + \sin(\psi)D_\eta M_\sigma \mathcal{V}^{-1} \bigl[ M_2^{-1,0} T_{-1,0} + M_2^{0,0} + M_2^{0,-1} T_{0,-1} \bigr] \Bigr) E_{Lag}.
    \end{aligned}  
  \end{equation}
\end{cor}

\subsection{Stability Analysis for the SD Method}\label{SDMOF}

We already stated that the asymptotic behavior of the solution of (\ref{ODE}) depends 
on the real parts of the eigenvalues of $S\equiv S(\alpha,\beta,\gamma)$ given by
(\ref{semidiscr}). For linear stability, we wish them to lie in the left half of the complex plane.
In case there are no such parameters $\alpha,\beta \in \mathbb{R}^+$ and $\gamma > \alpha +\beta -1$, we want 
at least to find a set of parameters $(\alpha,\beta,\gamma)$ with preferably small real parts.
If we denote the maximal real part over all eigenvalues corresponding to a certain
parameter set $(\alpha,\beta,\gamma)$ and a fixed set of test cases with respect to $w_x,w_y,\psi$ 
by $\Lambda \equiv \Lambda(\alpha,\beta,\gamma)$, this leads to the optimization problem
\begin{align}\label{OP}
  \underset{ \alpha,\beta \in \mathbb{R}^+ , \gamma > \alpha + \beta-1 }{\text{arg min} } \; \Lambda(\alpha,\beta,\gamma).
\end{align}
Note that this is a continuous and non-convex problem. Future studies could give a strict
treatment of the optimization problem (\ref{OP}), but since we just want to demonstrate
the influence of the chosen polynomial basis on the stability of the SD Method with and
without modal filtering, we will just  concentrate on certain subranges for $\alpha,\beta$ and $\gamma$. 
In the following, we will concentrate on the subrange of
  \begin{align*}
    P := \{ (\alpha,\beta,\gamma) \ | \ \alpha,\beta = 0.1,0.2,\dots,2 \text{ and } \gamma=\alpha+\beta,\alpha+\beta+0.1,\dots,6 \}
  \end{align*}
due to the observation that higher parameters lead to much worse condition numbers, see \ref{Konditionszahl}.
For a fixed parameter set $(\alpha,\beta,\gamma)$ in such a subrange we will compute all eigenvalues
of $S$ and look for the one with the greatest real part. This will be done for several test 
cases with respect to the flow direction $\psi \in [0,\frac{\pi}{2}]$ and the wave 
numbers $w_x,w_y \in [-\pi,\pi]$. Then we will again look after the greatest real part among 
all these test cases. The resulting eigenvalue and its real part will be taken as an evidence
for the stability of the underlying parameter set $(\alpha,\beta,\gamma)$ and its corresponding
polynomial basis. All of this is done in MatLab, where a descriptive pseudo code is given by 
Algorithm \ref{OF}.
\begin{algorithm}
\caption{without filtering}\label{OF}
\begin{algorithmic}[1]
\For{$\alpha,\beta=0.1:0.1:2, \; \gamma=\alpha+\beta:0.1:6$} 
  \State{$L<<0$ ( e.g. $-42$ )} 
  \For{$\psi=0:\frac{\pi}{8}:\frac{\pi}{2}, \; w_x,w_y=-\pi:\frac{\pi}{2}:\pi$} \Comment{test cases}
    \State compute S
    \State v = eig(S)
    \State v$_{re}$ = real(v)
    \State $\Lambda$ = max(v$_{re}$)
    \If {$\Lambda > L$} 
      \State $L=\Lambda$
    \EndIf
  \EndFor
  \State \Return $(\alpha,\beta,\gamma,L)$ \Comment{L is the greatest real part}
\EndFor
\end{algorithmic}
\end{algorithm}
The numerical results show that the maximal real part of all eigenvalues $L$ is
independent of $\alpha,\beta$ and $\gamma$. Hence the linear stability is independent of the 
corresponding basis of APK polynomials, which coincides with the following Theorem.

\begin{thm}
  In Lemma \ref{lemma} the eigenvalues of the semi-discretization $S$, given by (\ref{semidiscr}), 
  are independent of the basis of APK polynomials corresponding to $\alpha,\beta$ and $\gamma$.
\end{thm}
\begin{proof}
  Looking at $S$, only the differential matrices $D_\xi, D_\eta$, given by (\ref{diffM}),
  and the Vandermonde matrix $\mathcal{V}$, coming from the interpolation approach,
  depend on the underlying basis $\{ \phi_k \ | \ 1 \leq k \leq K_F \}$ of $\mathbb{P}_N(\mathbb{T})^2$.
  Hence the proof is done, when we can eliminate this matrices from $S$. Therefore note that 
  \begin{align*}
    D_\xi & = \Bigl( \partial_\xi \phi_k \bigl( T_{0,0}({\bf x}^s_j) \bigr) \Bigr)_{j,k}^{K_s,K_F} \\
      & = \bigl( \partial_\xi \phi_k(\boldsymbol{x}) |_{\boldsymbol{x}={\bf x}_j^s} \bigr)_{j,k} \\
      & = \Bigl( \partial_\xi \bigl[ \sum_{i=1}^{K_F} \phi_k({\bf x}^F_i) L_i(\boldsymbol{x}) \bigr]|_{\boldsymbol{x}={\bf x}_j^s} \Bigr)_{j,k} \\
      & = \bigl( \sum_{i=1}^{K_F} \phi_k({\bf x}^F_i) \partial_\xi L_i(\boldsymbol{x})|_{\boldsymbol{x}={\bf x}_j^s} \bigr)_{j,k} \\
      & = \bigl( \partial_\xi L_i({\bf x}_j^s) \bigr)_{j,i}^{K_s,K_F} \cdot \bigl( \phi_k({\bf x}^F_i) \bigr)_{i,k}^{K_F,K_F} \\
      & = D_\xi^{Lag} \mathcal{V}
  \end{align*}
  and analogous $D_\eta = D_\eta^{Lag} \mathcal{V}$ with
  \begin{displaymath}
    D_\xi^{Lag} := \bigl( \partial_\xi L_i({\bf x}_j^s) \bigr)_{j,i}^{K_s,K_F} \ , \ D_\eta^{Lag} := \bigl( \partial_\eta L_i({\bf x}_j^s) \bigr)_{j,i}^{K_s,K_F}
  \end{displaymath}
  independent of the basis $\{ \phi_k \}$. Hence
  \begin{align}\label{subst}
    D_\xi \mathcal{V}^{-1} = D_\xi^{Lag} \ , \ D_\eta \mathcal{V}^{-1} = D_\eta^{Lag}
  \end{align}
  and
  \begin{equation}
    \begin{aligned}\label{indep}
      S = & - \bigl( \cos(\psi) D_\xi^{Lag} [M_1^{-1,0}T_{-1,0} + M_1^{0,0} + M_1^{0,-1}T_{0,-1}] \\
	  & + \sin(\psi) D_\eta^{Lag} [M_2^{-1,0}T_{-1,0} + M_2^{0,0} + M_2^{0,-1}T_{0,-1}] \bigr) E_{Lag}
    \end{aligned}
  \end{equation}
  is independent of the underlying basis $\{ \phi_k \}$.
\end{proof}
Note that we have done the proof for an arbitrary basis $\{\phi_k\}$ of $\mathbb{P}_N(\mathbb{T})^2$ and that by (\ref{subst}) moreover the more general universal update scheme (\ref{schemeMatrix2}) for a general (non-linear) conservation law becomes
\begin{align*}
  \frac{\d {\bf u}^s}{\d t}(t) = - ( \xi_x D_\xi^{Lag} + \xi_y D_\eta^{Lag} ) {\bf F}_1(t) - ( 
\eta_x D_\xi^{Lag} + \eta_y D_\eta^{Lag} ) {\bf F}_2(t)
\end{align*}
with ${\bf F}_1(t),{\bf F}_2(t)$ independent of the basis $\{\phi_k\}$, when we use an interpolation approach to reconstruct the flux. This leads to the following corollary.
\begin{cor}
  The SD Method (without filtering) with a interpolation approach is independent of the underlying polynomial basis.
\end{cor}
\begin{rem}
 While the SD Method with an interpolation approach theoretically is independent of the underlying polynomial basis, 
 in computation their condition properties may have an influence on the 
resulting scheme, see Table \ref{Konditionszahl}.
 By substituting $D_\xi$ by $D_\xi^{Lag}$ and $D_\eta$ by $D_\eta^{Lag}$ and working with (\ref{indep}) in our implementation, 
 we blind out condition properties of the underlying basis and just focus on the pure stability properties of the SD Method. 
 However, one should note that this won't be the case in the next Subsection, where the Vandermonde matrix 
 and its inverse will not cancel out each other. Then, to cover condition issues as well, we focus on a proper subrange of parameters which 
 will ensure good condition numbers. In such a subrange we then will optimize the linear stability of the SD Method.
\end{rem}
Table \ref{realparts1} shows the maximal real parts among all test cases for certain orders $N$.
\begin{table}[!htb]
\begin{center}
 \begin{tabular}{c|c}
 $N$ &  maximal real part $L$  \\ \hline \hline
 2 & 5.228025e+00 \\
 3 & 7.671293e+00 \\
 4 & 1.360921e+01 \\
 5 & 2.010942e+01 \\
\end{tabular} 
\caption{maximal real parts}\label{realparts1}
\end{center}
\end{table}
These values show
that the SD Method on triangles is 
not stable, independent of the polynomial 
basis in the interpolation approach.
This was already stated by Van den Abeele,
Lacor and Wang in \cite{van2008stability}. 
As one can see in Table \ref{realparts1} the 
instabilities, i.e. Re$(\lambda)$, of the method 
can be 'quite high'. To deal with this problem we will apply the modal filter in the SD Method in the next Section.

\subsection{Numerical Stability Analysis for the SD Method with Modal Filtering}

As we proposed in Section \ref{Filter}, we use the SV Method seen as a spectral method with modal filtering to obtain milder instabilities in the von Neumann analysis and hence preserving the scheme to blow up at discontinuities. While we used $D\mathcal{V}^{-1} = D^{Lag}$ in Section \ref{SDMOF} to show that the SD Method is independent of the underlying basis this term becomes 
\begin{displaymath}
  DM_\sigma\mathcal{V}^{-1}
\end{displaymath}
in the semi discretization $S_\sigma$, given by (\ref{semidiscr2}), 
for the SD Method with modal filtering. This again leads to $S_\sigma \equiv 
S_\sigma(\alpha,\beta,\gamma)$ and the method to depend on the corresponding basis of APK 
polynomials. Following section \ref{SDMOF} we use Algorithm \ref{OF} with $S=S_\sigma$ for a 
(linear) stability analysis for the SD Method with modal filtering. But first, note that the values 
of the modal filter $\sigma$ in the filter matrix $M_\sigma:= \diag{\sigma_1,\dots,\sigma_{K_F}}$ 
depend on certain constants, which were explained in Section \ref{SVM}. For the following numerical 
results we set
\begin{align*}
  h = \frac{\sqrt{2}}{6} \, , \quad \text{see (\ref{viscosityStrength})}, \\
  C_{fix} = \frac{1}{2} \, , \quad \lambda_{\max} = 1 \, , \quad \text{see (\ref{timestep})}.
\end{align*}
For order $N=3$, different orders  $p\in \mathbb{N}_1^{5}$ and constants $c \in \mathbb{N}_2^{8}$ in the filter strength (\ref{viscosityStrength}), 
Table \ref{realparts2} - \ref{realparts4} shows the smallest real parts $L$ 
(maximum over all test cases) 
and their corresponding parameters $(\alpha,\beta, \gamma) \in P$.
\begin{table}[!htb]
\begin{subfigure}[!htb]{0.49\textwidth} 
\centering
\begin{tabular}{c|c|c|c}
    $p$ & $c$ & $(\alpha,\beta,\gamma)$ & $L$ \\ \hline \hline
    1 & 2 & (2,2,6) & 4.474928e+00 \\
      & 3 & (2,2,6) & 3.452260e+00 \\
      & 4 & (2,2,6) & 2.825792e+00 \\
      & 5 & (2,2,6) & 2.341697e+00 \\
      & 6 & (2,2,4.2) & 1.949950e+00 \\
      & 7 & (2,2,4.5) & 1.630852e+00 \\
      & 8 & (2,2,4.6) & 1.378782e+00 \\
\end{tabular}
\end{subfigure}%
  ~
\begin{subfigure}[!htb]{0.49\textwidth} 
\centering
\begin{tabular}{c|c|c|c}
  $p$ & $c$ & $(\alpha,\beta,\gamma)$ & $L$ \\ \hline \hline
    2 & 2 & (2,2,5.2) & 2.500718e+00 \\
      & 3 & (2,2,5) & 2.000045e+00 \\
      & 4 & (2,2,4.7) & 1.783495e+00 \\
      & 5 & (2,2,5.7) & 1.576658e+00 \\
      & 6 & (2,2,5.5) & 1.398979e+00 \\
      & 7 & (2,2,5.3) & 1.242022e+00 \\
      & 8 & (2,2,5.1) & 1.111981e+00 \\
\end{tabular}
\end{subfigure}%
   \caption{maximal real parts for $p=1$ and $p=2$}\label{realparts2}
\end{table}
\begin{table}[!htb]
\begin{subfigure}[!htb]{0.49\textwidth} 
\centering
\begin{tabular}{c|c|c|c}
    $p$ & $c$ & $(\alpha,\beta,\gamma)$ & $L$ \\ \hline \hline
    3 & 2 & (2,2,4.6) & 2.137056e+00 \\
      & 3 & (2,2,5.5) & 1.826333e+00 \\
      & 4 & (2,2,5.2) & 1.560388e+00 \\
      & 5 & (1.1,1.1,6) & 1.349429e+00 \\
      & 6 & (1.1,1.1,6) & 1.181007e+00 \\
      & 7 & (1,1.1,6) & 1.050789e+00 \\
      & 8 & (1,1,6) & 9.382501e-01 \\ \hline
\end{tabular}
\end{subfigure}%
  ~
\begin{subfigure}[!htb]{0.49\textwidth} 
\centering
\begin{tabular}{c|c|c|c}
    $p$ & $c$ & $(\alpha,\beta,\gamma)$ & $L$ \\ \hline \hline
    4 & 2 & (2,2,5.4) & 1.836013e+00 \\
      & 3 & (2,2,5.1) & 1.511921e+00 \\
      & 4 & (2,2,4.7) & 1.339156e+00 \\
      & 5 & (1,1,5.9) & 1.187603e+00 \\
      & 6 & (1,1,5.9) & 1.094414e+00 \\
      & 7 & (1,1,5.9) & 1.029666e+00 \\
      & 8 & (1,1,5.9) & 9.742486e-01 \\ \hline
\end{tabular}
\end{subfigure}%
   \caption{maximal real parts for $p=3$ and $p=4$}\label{realparts3}
\end{table}
\begin{table}[!htb] 
\centering
\begin{subfigure}[!htb]{0.49\textwidth}
\begin{tabular}{c|c|c|c}
    $p$ & $c$ & $(\alpha,\beta,\gamma)$ & maximal real part $L$ \\ \hline \hline
    5 & 2 & (2,2,5) & 1.525311e+00 \\
      & 3 & (1,1,5.9) & 1.355060e+00 \\
      & 4 & (2,2,4.3) & 1.186385e+00 \\
      & 5 & (2,2,4.1) & 1.075861e+00 \\
      & 6 & (2,2,4) & 1.100214e+00 \\
      & 7 & (1,1,5.9) & 1.097985e+00 \\
      & 8 & (1,1,5.9) & 1.056802e+00   
\end{tabular}
\end{subfigure}%
   \caption{maximal real parts for $p=5$}\label{realparts4}
\end{table}
We can see in Tables \ref{realparts2} - \ref{realparts4} that one can obtain 
much milder instabilities of the method by modal filtering, 
where also the choice of family of APK polynomials for the underlying basis of $\mathbb{P}_N(\mathbb{T})^2$ has an influence.
Note that in fact both $p$, $c$ and every parameter $\alpha, \beta$ and $\gamma$ effect the linear stability. 
Approximately, the parameter tuples $(1,1,6)$ and $(2,2,5)$ appear promising. 
However, by considering their conditional numbers in Table 
\ref{Konditionszahl}, the parameter tuple $(1,1,6)$ becomes unreasonable, due to 
its high condition numbers. 
Thus, we will focus on parameter tuple $(2,2,5)$ in the following example which will help us to find suitable choices for 
the parameter $c$ and the filter order $p$.
\begin{ex}
  First, we observe that the constant $c$ has an influence on the linear stability by comparing the parameter tuple $(2,2,5)$ for 
  $N=3$, $p=2$ and different $c$. In this case we have
  \begin{align*}
    L = & 2.502205e+00 \text{ for } c=2, \\
    L = & 2.000045e+00 \text{ for } c=3, \\
    L = & 1.824525e+00 \text{ for } c=4, \\
    L = & 1.639219e+00 \text{ for } c=5, \\
    L = & 1.459258e+00 \text{ for } c=6, \\
    L = & 1.290925e+00 \text{ for } c=7, \\
    L = & 1.135887e+00 \text{ for } c=8,
  \end{align*}
  which indicates that by increasing parameter $c$ the instability gets milder.
  This coincides with Theorem \ref{Spekralersatz} which states the equivalence of modal filtering by our natural filters 
  and the SV Method. 
  In this formulation $c$ is proportional to $\varepsilon_N$, see (\ref{viscosityStrength}), 
  which again is proportional to the dissipation we add to the underlying conservation law by the viscosity term. \\
  
  In the same way, by comparing the parameter tuple $(2,2,5)$ for $N=3$ and $c=8$, we can observe the influence 
  of the parameter $p$. Here
  \begin{align*}
    L = & 1.406930e+00 \text{ for } p=1, \\
    L = & 1.135887e+00 \text{ for } p=2, \\
    L = & 1.227645e+00 \text{ for } p=3, \\
    L = & 1.481018e+00 \text{ for } p=4, \\
    L = & 1.593715e+00 \text{ for } p=5,  
  \end{align*}
  suggest $p=2$ to be the best choice. \\
  
  Finally, we observe the influence of the parameter $\gamma$ by comparing parameter tuples $(2,2,\gamma)$ for 
  $N=3$, $c=8$ and $p=2$. Here
  \begin{align*}
    L = & 1.284659e+00 \text{ for } \gamma=4, \\
    L = & 1.212759e+00 \text{ for } \gamma=4.5, \\
    L = & 1.135887e+00 \text{ for } \gamma=5, \\
    L = & 1.200003e+00 \text{ for } \gamma=5.5, \\
    L = & 1.216268e+00 \text{ for } \gamma=6 
  \end{align*}
  show the influence.
\end{ex}
The natural filter of $(1,1,2)$, the one of $(2,2,5)$ and the shaped raised 
cosine filter of $8$th order can be seen and compared in Figure 
\ref{filter_plots}.

\begin{figure}[!htb]
  \centering
  \begin{subfigure}[b]{0.9\textwidth}
    \includegraphics[width=\textwidth]{%
      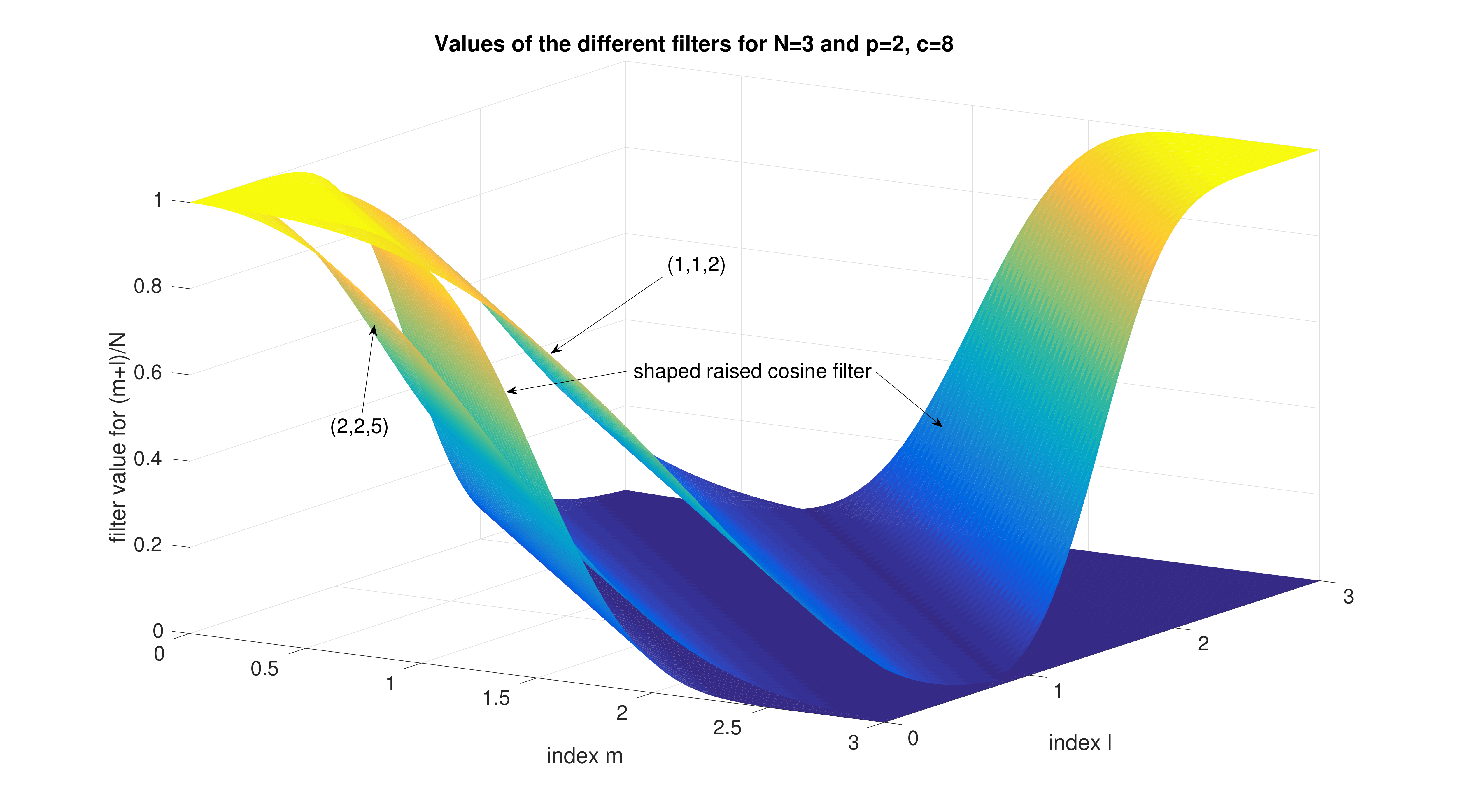} 
  \end{subfigure}%
  \\
  \begin{subfigure}[b]{0.9\textwidth}
    \includegraphics[width=\textwidth]{%
      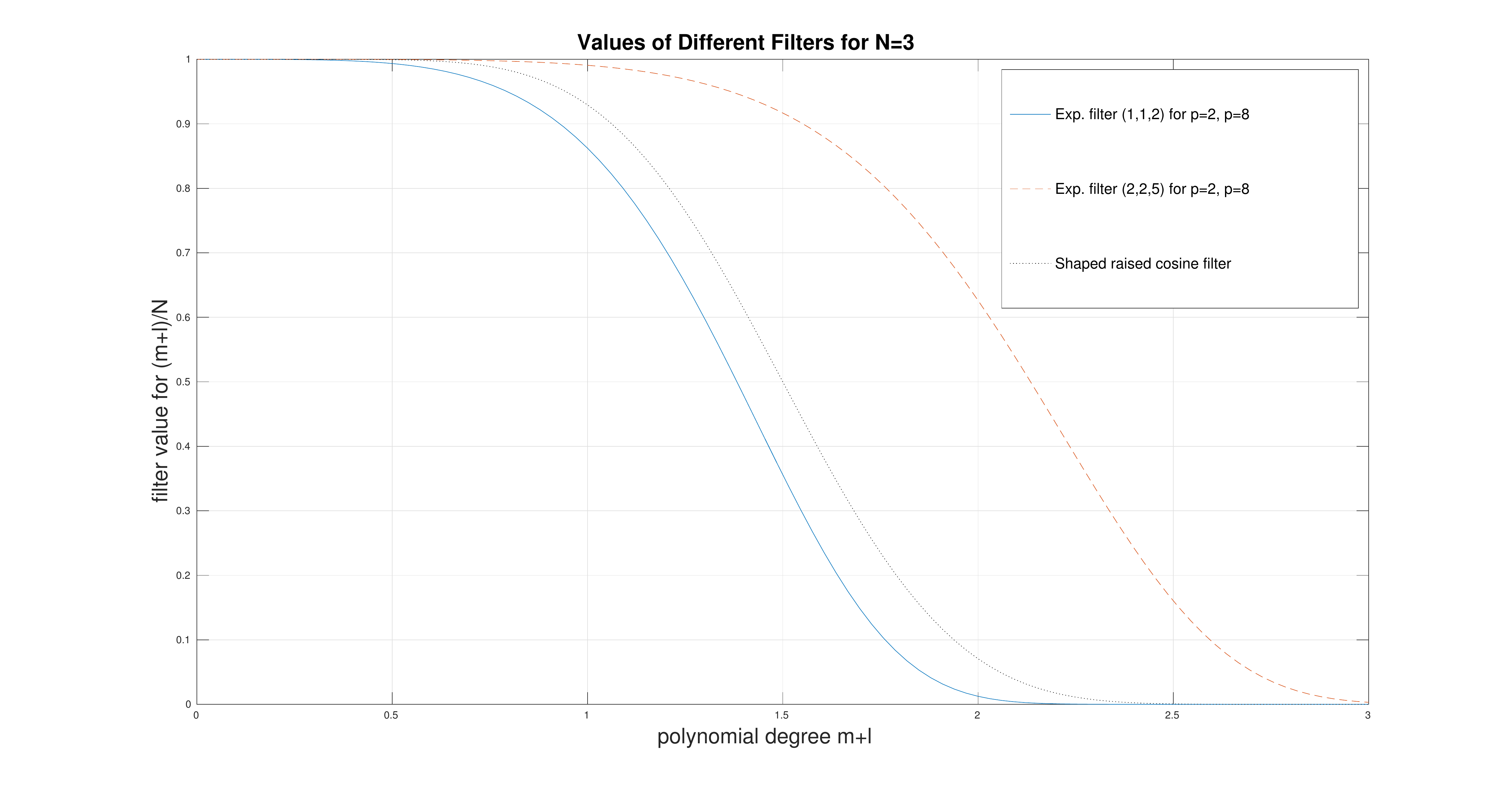}
  \end{subfigure}%
  \caption{$(\alpha,\beta,\gamma)=(1,1,2)$, $(2,2,5)$ with their natural filters and the shaped raised cosine filter} \label{filter_plots}
\end{figure}

We end this Section by recommending the parameter tuple $(\alpha,\beta, \gamma)=(2,2,5)$ with $c=8$ and filter order $p=2$ for the polynomial degree $N=3$. 
The following Section will provide numerical test which build around this choices.
\section{Numerical tests} \label{Numerical Test}

After presenting our theoretical results we want to give a short numerical investigation to show that our conclusions are justified.
We consider the two dimensional Burgers' equation
\begin{equation*}
 u_t(x,y,t)+u(x,y,t)\l( u_x(x,y,t)+u_y(x,y,t)\r)=0
\end{equation*}
in the domain $[-1,1]^2$ with the initial condition
\begin{equation*}
 u_0(x,y)=\frac{1}{4}+\frac{1}{2} \sin \l( \pi(x+y)\r)
\end{equation*}
and periodic boundary conditions $u(-1,y,t)=u(1,y,t)$ and $u(x,-1,t)=u(x,1,t)$. 
At $t=0.5s$ two discontinuities  arise at $y=\frac{3}{2}-x$ and $y=\frac{5}{2}-x$. We use 
$1088$ triangles for the spatial discretization and present the numerical solutions at $t=0.45s$ for several parameter selections. 
Without filtering high oscillations already develop in much earlier calculations and the SD Method collapses. \\
In Tables \ref{Burgerstabelle1}-\ref{Burgerstabelle3} the minimum and maximum 
of the numerical solutions for filter parameters $p$ and $c$ are shown in the 
polynomial cases 
$(\alpha,\beta,\gamma)= (1,1,2)$ and $(2,2,5) $. These values give a first indication for the stability of the resulting method. 

\begin{table}
\begin{subfigure}[b]{0.49\textwidth} 
  \centering
  \begin{tabular}{c|c|c|c}
    $p$  & $c$ & min & max  \\ \hline \hline
    $2 $ & $8$ & -3.24 & 6.00 \\
    $3$  & $ 8 $& -6.55 & 3.67  \\
    $4$  &$6 $ & -12.32 & 6.03 \\
    $4$  & $8 $ & -11.08 & 4.96 \\
  \end{tabular} 
  \caption{$ (\alpha,\beta,\gamma)=(1,1,2 ),\;$  $t=0.4s$}
  \label{Burgerstabelle1}
\end{subfigure}%
  ~
\begin{subfigure}[b]{0.49\textwidth} 
  \centering
  \begin{tabular}{c|c|c|c}
    $p$  & $c$ & min & max  \\ \hline \hline
    $2 $ & $8$ & -1.55 & 2.46 \\
    $3$  & $ 8 $& -6.54 & 4.23  \\
    $4$  &$6 $ & -3.10 & 5.98  \\
    $4$  & $8 $ & -2.44 & 5.64 \\
  \end{tabular} 
  \caption{$ (\alpha,\beta,\gamma)=(2,2,5 ),\;$  $t=0.45s$}
  \label{Burgerstabelle3}
\end{subfigure}
\end{table}

The polynomial degree $N$  in these test-cases is always three. 
For most parameters $p$, $c$ and $(\alpha,\beta,\gamma)$, we were not able to calculate in time further than $0.4s$.
However, for the choice $p=2$ and $c=8$ motivated by our stability analysis we were indeed able to calculate 
till $0.45$ for both parameters $(1,1,2)$ and $(2,2,5)$.
In Figure \ref{Abbildung1} the 2d and 3d plots of the numerical solutions are 
presented. 
No post-processing was applied.

\begin{figure}[!htb]
  {{\includegraphics[width=0.24\textwidth]{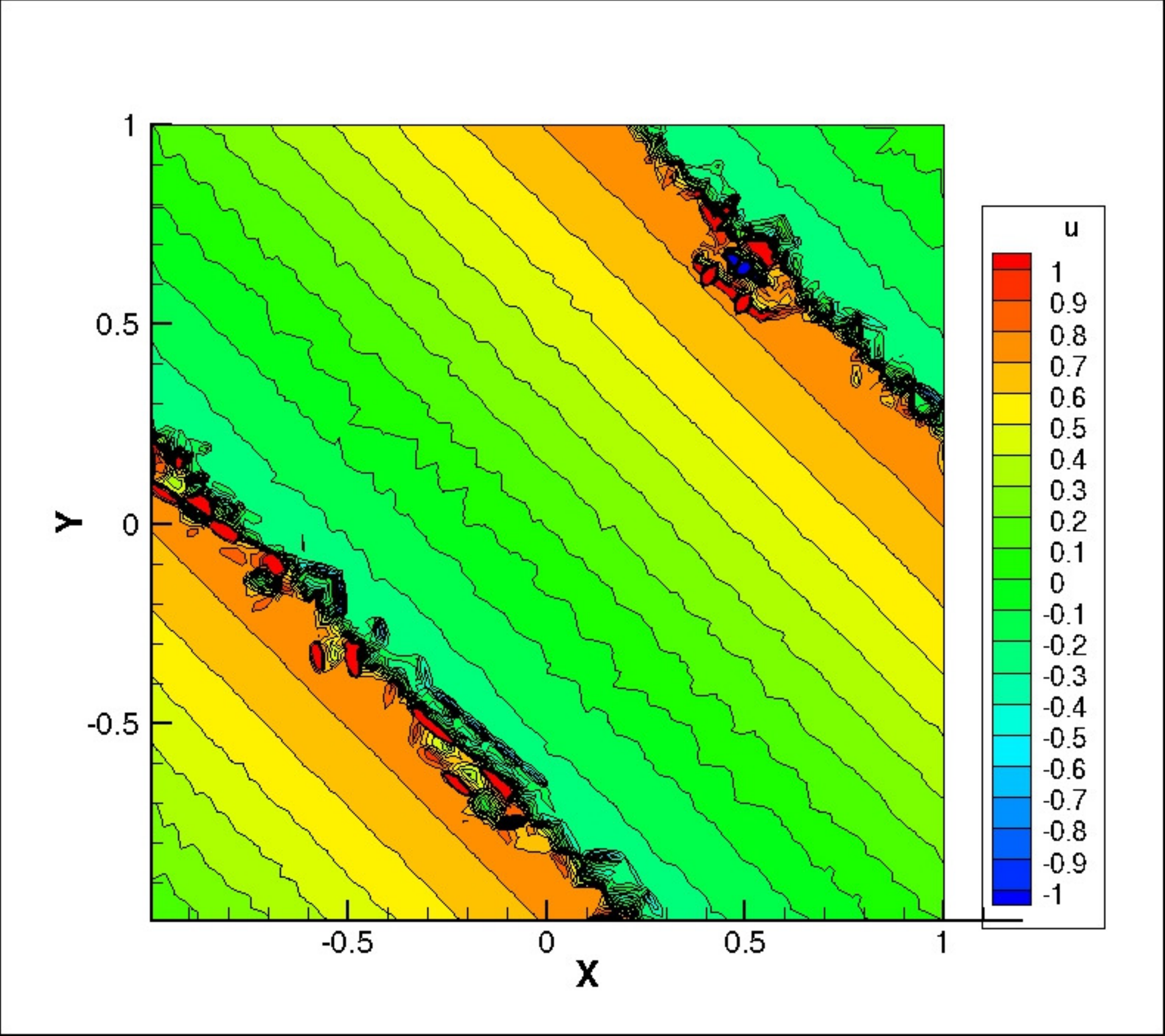}}
  {\includegraphics[width=0.24\textwidth]{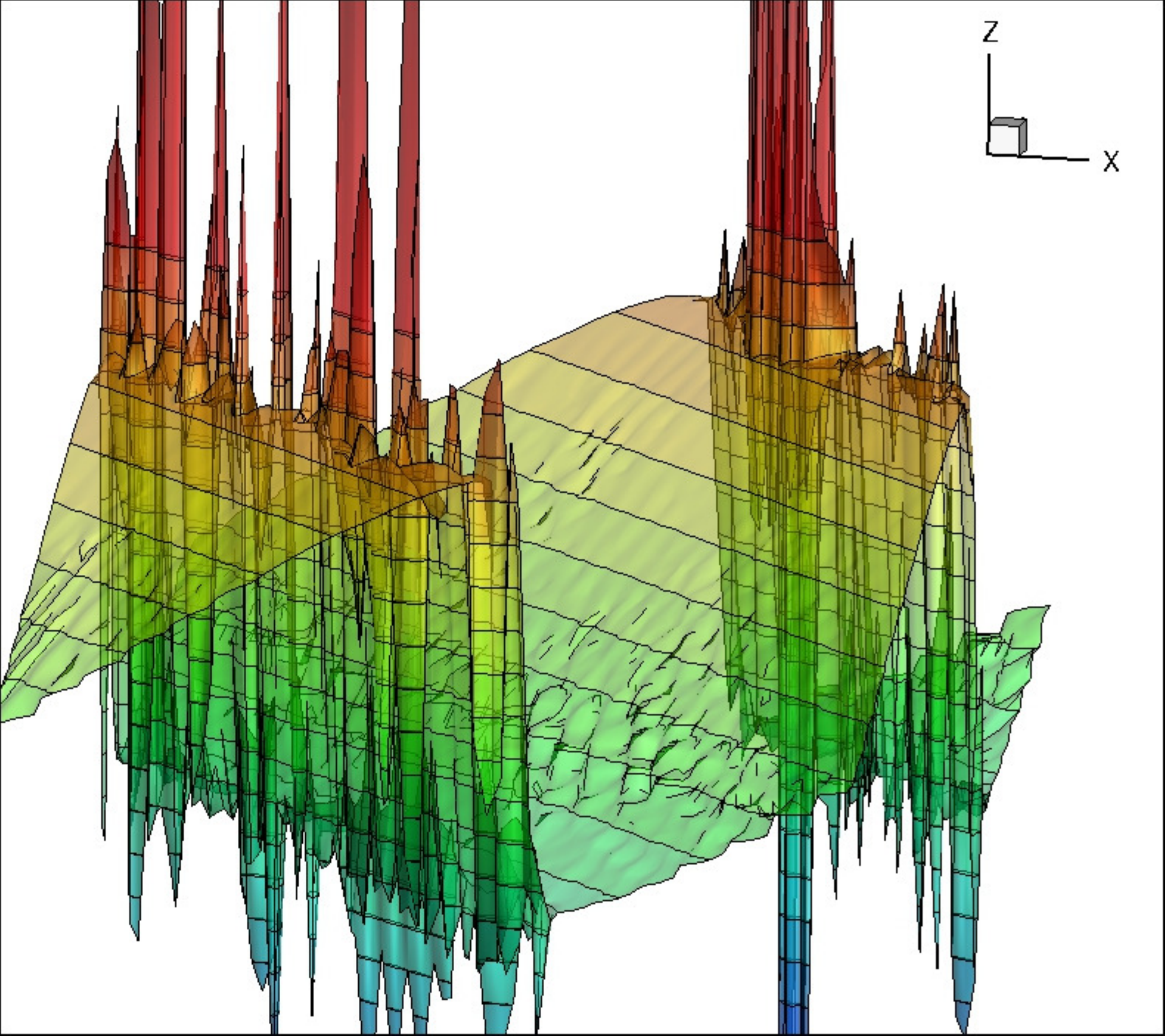}}
  {\includegraphics[width=0.24\textwidth]{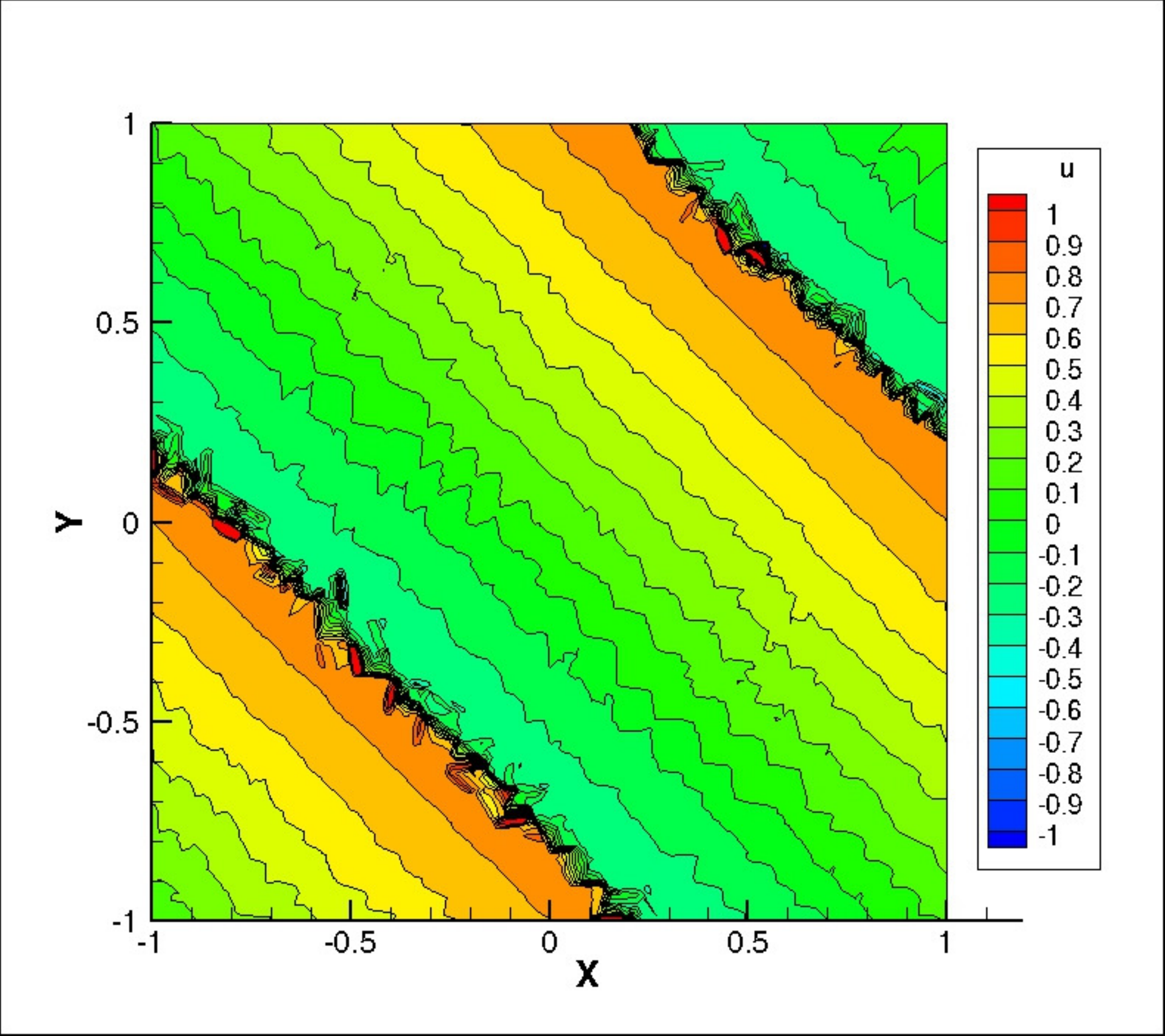}}
  {\includegraphics[width=0.24\textwidth]{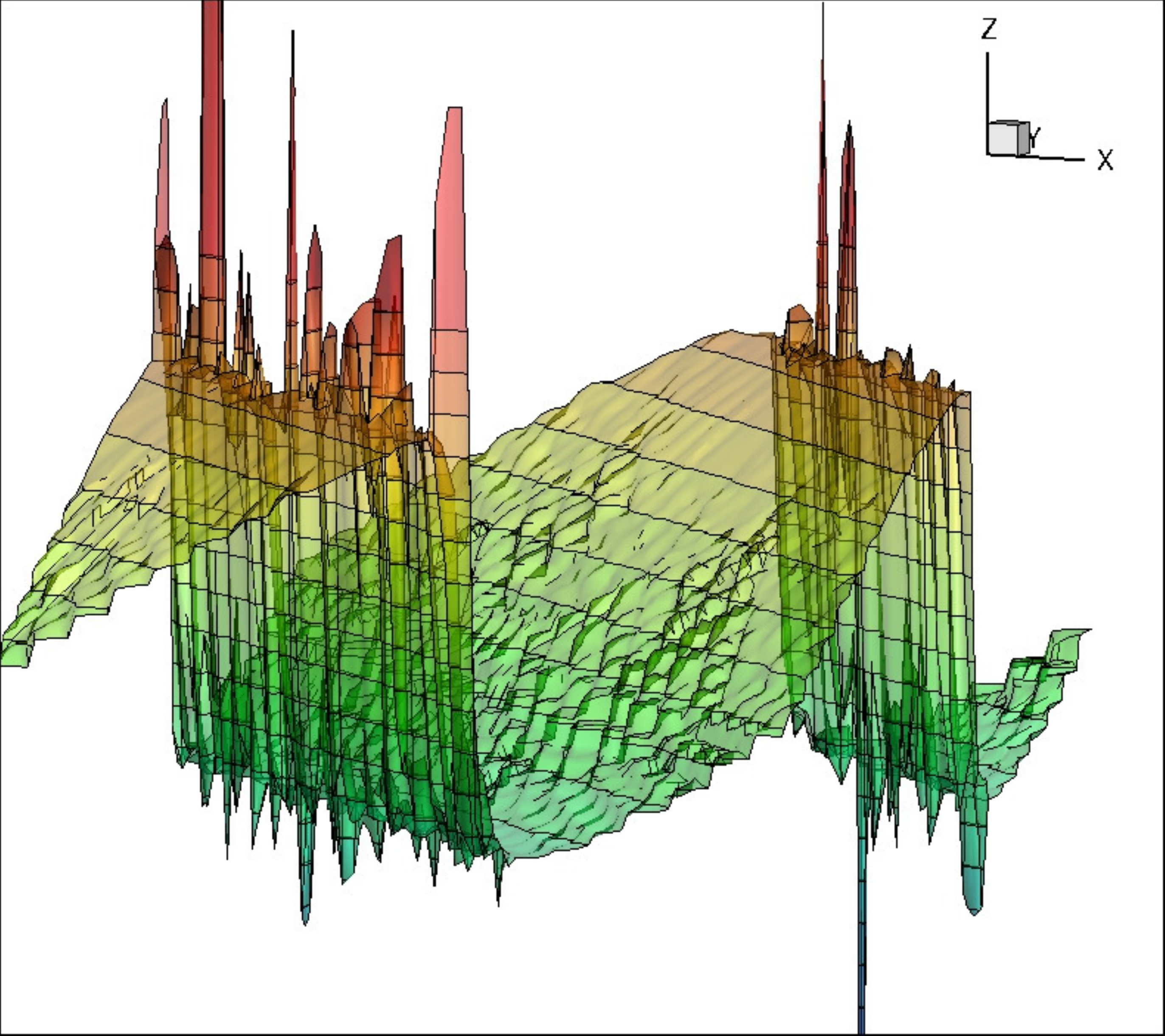}}}
  \text{\hspace{1.3cm}  $(\alpha,\beta,\gamma)=(1,1,2),\; t=0.45s$ \hspace{1.8cm} $(\alpha,\beta,\gamma)=(2,2,5),\; t=0.45s$ }
  \caption{  $N=3$,  $p=2$,  $C=8$} \label{Abbildung1}
\end{figure}

Predicted by the stability analysis in Section \ref{Stability} and 
coinciding with Tables \ref{Burgerstabelle1} and \ref{Burgerstabelle3}, 
the selection of parameters $p=2$ and $c=8$ using the APK polynomials $(2,2,5)$
and their natural filters show significant stronger stability properties than 
for instance $(1,1,2)$.
Here, two attributes play a key role in the calculation.
First $(2,2,5)$ has a small condition number, see Table \ref{Konditionszahl}. 
But second and quiet more important are the good properties - using the natural filter with $\gamma=5$ - 
from point of \emph{linear} stability, suggested by the stability analysis in Section \ref{Stability}.
Following this observation, we wish to combine a small condition number like for $(1,1,2)$ and 
the good stability of the choice $(\alpha,\beta,\gamma)=(2,2,5)$ with $p=2$ and $c=8$.
Therefor we next observed the polynomial basis of $(1,1,2)$, however equipped with the natural filter of 
$(2,2,5)$ instead of it's own one. 
I.e we choose $\gamma=5$ in \eqref{Filterapproximiert} instead of the natural choice $\gamma=2$.
The resulting numerical solutions, again presented in a 2d and 3d plot, 
can be seen in Figure \ref{Abbildung2} and once more show an improvement.

\begin{figure}[!htb]
{{\includegraphics[width=0.5\textwidth]{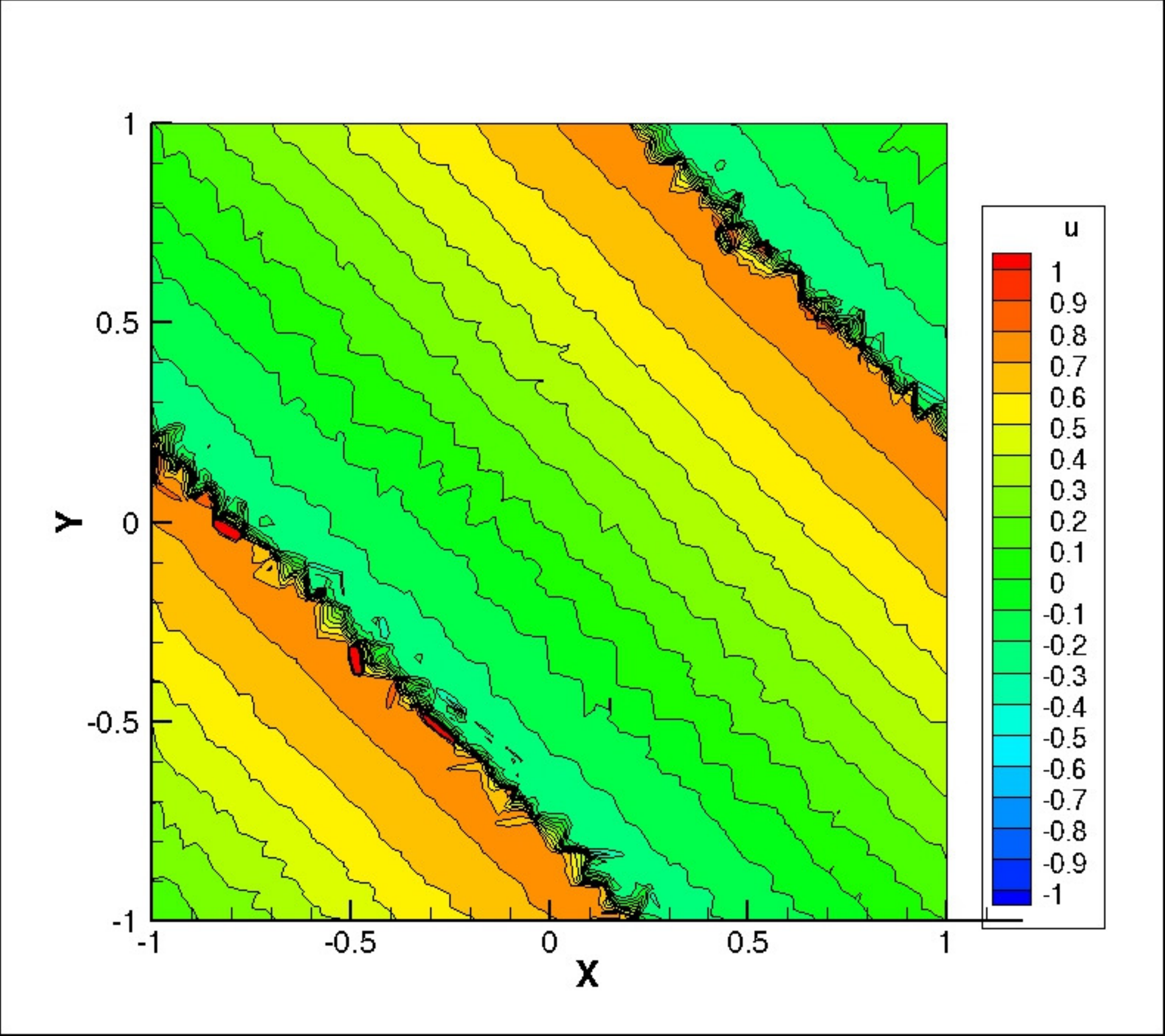}}{
\includegraphics[ width=0.5\textwidth]{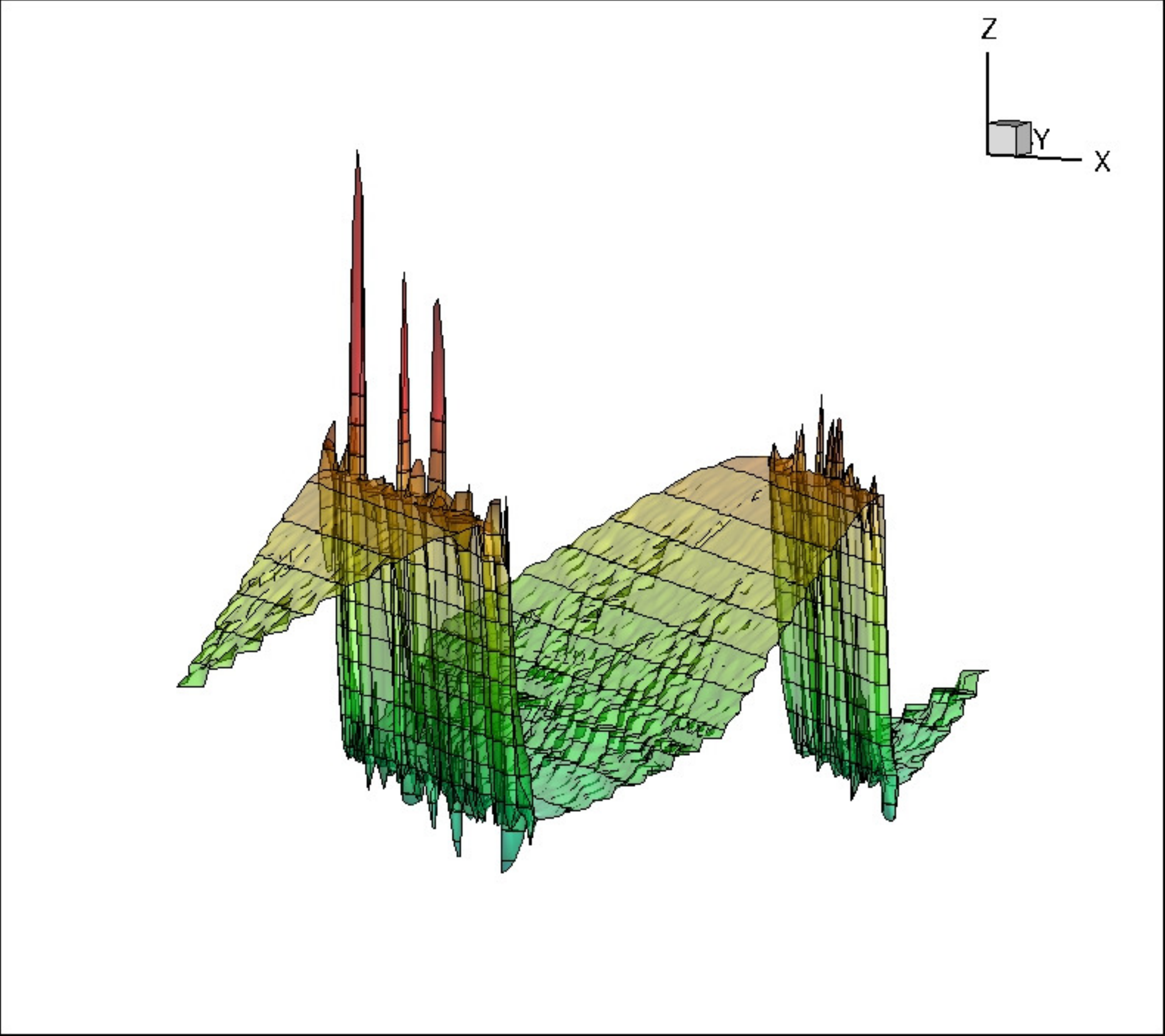}}}
\text{\hspace{2.5cm}   $t = 0.45s$ \hspace{5cm} $t = 0.45s$ }
\caption{$(\alpha,\beta,\gamma)=(1,1,2)$ with the natural filter of $p=2$, $c=8$ and $(2,2,5)$} \label{Abbildung2}
\end{figure}

\section{Discussion and Conclusion}
We started this paper by extending the SD method by more general bases of APK polynomials 
instead of Lagrange or PKD polynomials and instantly observed their numerical qualification by 
their condition numbers. 
The classical SD method was known to be unstable on triangular elements and unfortunately we were able 
to obtain the same result for every extended version of the SD method.
Hence addressing the problem of Gibb's phenomenon, which often leads unstable schemes to blow up, 
we applied the well known SV method to milder instabilities. 
Due to the APK polynomials to fulfill certain eigenvalue problems, we derived an equivalent but 
from point of computational costs much more efficient formulation of the SV method via 
modal filtering by natural exponential filters with respect to the polynomial basis. 
Whole new error estimates for filtered APK extensions for smooth functions then were given. 
We also gave a von Neumann analysis for linear stability for the extended SD method with modal filtering 
by natural filters of APK polynomials. 
The idea was for the right choice of APK polynomials to may have some positive influence on the 
stability and accuracy of the method.
To the best of our knowledge, this work was the first to give such an analysis for a scheme with 
spectral filtering.
This stability analysis indeed suggested a certain choice of APK polynomials and parameters in their 
natural filters, which then were observed in a short numerical investigation.
The numerical solutions we obtained show significant stronger stability properties than the 
ones resulting from classical SD methods. 
By combining our results from the analysis of condition numbers and linear stability for different 
APK polynomials, we equipped basis polynomials with favorable condition numbers with modal filters coming 
from a polynomial basis with strong stability properties.
Once more, this led to a significant improvement of the numerical solutions.

Despite this, it should also be said that Flux Reconstruction (FR) or Correction Procedure via 
Reconstruction (CPR) methods nowadays are getting more and more interesting, due to their strong 
stability properties, while SD methods are receiving decreasing attention.
We still analyzed modal filtering based on SD methods to investigate the pure influence of the 
orthogonal polynomials and their natural filters. 
In the FR/CPR approach a correction term is applied which might cancel out effects of the filter.
However, the observations made for spectral filtering in this paper target a first step to both 
a better understanding of the influence of spectral filtering and in future to gainfully adapt these ideas to 
FR/CPR methods. 
Improvements might be higher CFL numbers, thus larger time steps and more efficient methods, as well as 
more advanced methods in post-processing.

In future researches further developments for modal filtering itself would be favorable. 
Due to high condition numbers of the Vandermonde matrix for several bases of APK polynomials, 
we had to restrict ourself to certain subranges of parameter families. 
However, it could be possible to find filters with even better properties outside of these subranges. 
To bypass the problem of quite bad condition numbers in this new ranges, it would be promising to work 
with projection approaches in the series expansion, instead of interpolation ones. 

\section{Appendix:}

$M_i^{m,n}$ for the SD Method of 3rd order ($K_s=6,K_F=10$):
\begin{align*}
  M_1^{-1,0} & = \diag{E_{4 \times 4},0_{6 \times 6}}, \\
  M_1^{0,0} & = \diag{0_{5 \times 5},E_{2 \times 2},0_{1 \times 1},E_{2 \times 2}}, \\
  M_1^{0,-1} & = \diag{0_{4 \times 4},E_{1 \times 1},0_{2 \times 2},E_{1 \times 1},0_{2 \times 2}}, 
\\
  M_2^{-1,0} & = \diag{0_{1 \times 1},E_{2 \times 2},0_{7 \times 7}}, \\
  M_2^{0,0} & = \diag{0_{3 \times 3},E_{1 \times 1},0_{1 \times 1},E_{2 \times 2},0_{1 \times 
1},E_{1 \times 1},0_{1 \times 1}}, \\
  M_2^{0,-1} & = \diag{E_{1 \times 1},0_{3 \times 3},E_{1 \times 1},0_{2 \times 2},E_{1 \times 
1},0_{1 \times 1},E_{1 \times 1}}.
\end{align*}

$M_i^{m,n}$ for the SD Method of 4th order ($K_s=10,K_F=15$):
\begin{align*}
  M_1^{-1,0} & = \diag{E_{5 \times 5},0_{10 \times 10}}, \\
  M_1^{0,0} & = \diag{0_{6 \times 6},E_{3 \times 3},0_{1 \times 1},E_{2 \times 2},0_{1 \times 
1},E_{2 \times 2}}, \\
  M_1^{0,-1} & = \diag{0_{5 \times 5},E_{1 \times 1},0_{3 \times 3},E_{1 \times 1},0_{2 \times 
2},E_{1 \times 1},0_{2 \times 2}}, \\
  M_2^{-1,0} & = \diag{0_{1 \times 1},E_{3 \times 3},0_{11 \times 11}}, \\
  M_2^{0,0} & = \diag{0_{4 \times 4},E_{1 \times 1},0_{1 \times 1},E_{3 \times 3},0_{1 \times 
1},E_{2 \times 2},0_{1 \times 1},E_{1 \times 1},0_{1 \times 1}}, \\
  M_2^{0,-1} & = \diag{E_{1 \times 1},0_{4 \times 4},E_{1 \times 1},0_{3 \times 3},E_{1 \times 
1},0_{2 \times 2},E_{1 \times 1},0_{1 \times 1},E_{1 \times 1}}.
\end{align*}

$M_i^{m,n}$ for the SD Method of 5th order ($K_s=15,K_F=21$):
\begin{align*}
  M_1^{-1,0} & = \diag{E_{6 \times 6},0_{15 \times 15}}, \\
  M_1^{0,0} & = \diag{0_{7 \times 7},E_{4 \times 4},0_{1 \times 1},E_{3 \times 3},0_{1 \times 
1},E_{2 \times 2},0_{1 \times 1},E_{2 \times 2}}, \\
  M_1^{0,-1} & = \diag{0_{6 \times 6},E_{1 \times 1},0_{4 \times 4},E_{1 \times 1},0_{3 \times 
3},E_{1 \times 1},0_{2 \times 2},E_{1 \times 1},0_{2 \times 2}}, \\
  M_2^{-1,0} & = \diag{0_{1 \times 1},E_{4 \times 4},0_{16 \times 16}}, \\
  M_2^{0,0} & = \diag{0_{5 \times 5},E_{1 \times 1},0_{1 \times 1},E_{4 \times 4},0_{1 \times 
1},E_{3 \times 3},0_{1 \times 1},E_{2 \times 2},0_{1 \times 1},E_{1 \times 1},0_{1 \times 1}}, \\
  M_2^{0,-1} & = \diag{E_{1 \times 1},0_{5 \times 5},E_{1 \times 1},0_{4 \times 4},E_{1 \times 
1},0_{3 \times 3},E_{1 \times 1},0_{2 \times 2},E_{1 \times 1},0_{1 \times 1},E_{1 \times 1}}.
\end{align*}

\bibliographystyle{abbrv}
\bibliography{literature}

\end{document}